\documentclass{amsart}
\usepackage[nobysame]{amsrefs}
\usepackage[usesnames,svgnames]{xcolor}
\usepackage[sc]{mathpazo}
\usepackage[T1]{fontenc}
\usepackage{tikz}
\usepackage[colorlinks=true,
	urlcolor=MidnightBlue,
	citecolor=DarkGreen]{hyperref}
\usepackage{enumerate}
\newtheorem{theorem}{Theorem}[section]
\newtheorem{proposition}[theorem]{Proposition}
\newtheorem{conjecture}[theorem]{Conjecture}
\newtheorem{lemma}[theorem]{Lemma}
\newtheorem{corollary}[theorem]{Corollary}
\newtheorem{definition}[theorem]{Definition}
\theoremstyle{remark}
\newtheorem{remark}[theorem]{Remark}
\newtheorem{example}[theorem]{Example}
\newcommand{\alert}[1]{{\color{DarkGreen}\emph{#1}}}
\newcommand{\nc}[1]{N\!C_{#1}}
\newcommand{\gnc}[2]{N\!C_{#1}^{(#2)}}
\newcommand{\bnc}[2]{\overline{N\!C}_{#1}^{(#2)}}
\newcommand{\fix}[1]{\text{Fix}(#1)}
\newcommand{\cat}[2]{\text{Cat}^{(#2)}(#1)}
\newcommand{\colint}[2]{#1^{(#2)}}
\newcommand{\colref}[3]{\bigl(\!\bigl(\colint{#1}{0}\;\colint{#2}{#3}\bigr)\!\bigr)}
\newcommand{\tp}{\mathsf{T}}
\newcommand{\vb}{\mathbf{v}}
\newcommand{\ie}{\text{i.e.}\;}
\title[EL-Shellability and Noncrossing Partitions]{EL-Shellability and 
  Noncrossing Partitions Associated with Well-Generated Complex Reflection Groups}
\author{Henri M{\"u}hle}
\address{Fak. f{\"u}r Mathematik, Universit{\"a}t Wien, Oskar-Morgenstern-Platz 1, 1090 Vienna, Austria}
\email{henri.muehle@univie.ac.at}
\thanks{This work was funded by the FWF Research Grant No. Z130-N13.}

\begin{document}

\allowdisplaybreaks
	
\begin{abstract}
	In this article we prove that the lattice of noncrossing partitions is EL-shellable when associated with the
	well-generated complex reflection group of type $G(d,d,n)$, for $d,n\geq 3$, or with the exceptional 
	well-generated complex reflection groups which are no real reflection groups. This result was previously 
	established for the real reflection groups and it can be extended to the well-generated complex reflection group of
	type $G(d,1,n)$, for $d,n\geq 3$, as well as to three exceptional groups, namely $G_{25},G_{26}$ and $G_{32}$, using
	a braid group argument. We thus conclude that the lattice of noncrossing partitions of any well-generated
	complex reflection group is EL-shellable. Using this result and a construction by Armstrong and Thomas, we 
	conclude further that the poset of $m$-divisible noncrossing partitions is EL-shellable for every well-generated 
	complex reflection group. Finally, we derive results on the M{\"o}bius function of these posets previously 
	conjectured by Armstrong, Krattenthaler and Tomie. 
\end{abstract}

\maketitle

\section{Introduction}
  \label{sec:introduction}
In a seminal paper \cite{kreweras72sur}, Kreweras investigated noncrossing set partitions under refinement order. 
They quickly became a popular research topic and many interesting connections to other mathematical branches,
such as algebraic combinatorics, group theory, topology, and representation theory, have been found. For a survey
on the connection of noncrossing partitions with these mathematical branches, see 
\cites{mccammond06noncrossing,simion00noncrossing}. Many of these connections were made possible by regarding 
noncrossing set partitions as elements of the intersection poset of the braid arrangement. This observation 
eventually allowed for associating analogous lattices with every well-generated complex reflection group $W$, which
we denote by $\nc{W}$, and we call their elements the \alert{$W$-noncrossing partitions}. Meanwhile, 
these noncrossing partitions have been further generalized to so-called 
\alert{$m$-divisible $W$-noncrossing partitions}, and the corresponding poset will be denoted by $\gnc{W}{m}$, see 
\cites{armstrong09generalized,bessis11cyclic}. Kreweras' initial objects are obtained as the special case where
$W$ is the symmetric group and $m=1$. 

The main purpose of this article is to prove that the lattice of noncrossing partitions associated with a well-generated
complex reflection group is EL-shellable. The fact that a poset is EL-shellable implies a number of algebraic,
topological and combinatorial properties. For instance the Stanley-Reisner ring associated with an EL-shellable poset
is Cohen-Macaulay. For further implications of EL-shellability, see the end of Section~\ref{sec:shellability} or
\cites{bjorner96shellable,bjorner97shellable}. In particular, we prove the following theorem.

\begin{theorem}\label{thm:shellability_main}
	The lattice $\nc{W}$ of $W$-noncrossing partitions is EL-shellable when $W=G(d,d,n)$ for $d,n\geq 3$, or when 
	$W$ is an exceptional well-generated complex reflection group which is not a real reflection group. 
\end{theorem}

We recall in Section~\ref{sec:reflection_groups} that there are three infinite families of irreducible 
well-generated complex reflection groups, namely $G(1,1,n)$ for $n\geq 1$, $G(d,1,n)$ for $n\geq 1,d\geq 2$, and 
$G(d,d,n)$ for $n,d\geq 2$, as well as 26 exceptional well-generated complex reflection groups. Among these, the
following groups are real reflection groups: 
\begin{itemize}
	\item the group $G(1,1,n)$ for $n\geq 1$ is isomorphic to the Coxeter group $A_{n-1}$,
	\item the group $G(2,1,n)$ for $n\geq 2$ is isomorphic to the Coxeter group $B_{n}$,
	\item the group $G(2,2,n)$ for $n\geq 4$ is isomorphic to the Coxeter group $D_{n}$,
	\item the group $G(d,d,2)$ for $d\geq 3$ is isomorphic to the Coxeter group $I_{2}(d)$,
	\item the group $G(2,2,3)$ is isomorphic to the Coxeter group $A_{3}$, and
	\item the group $G(2,2,2)$ is isomorphic to the reducible Coxeter group $A_{1}\times A_{1}$,
\end{itemize}
see \cite{lehrer09unitary}*{Example~2.11} or \cite{broue98complex}*{Tables~1~and~2}. Six of the 26 exceptional 
irreducible well-generated complex reflection groups are exceptional real reflection groups, and we will list them
in Section~\ref{sec:shellability_exceptional}. The EL-shellability of noncrossing partition lattices associated with 
real reflection groups has been proven (in a uniform way!) by Athanasiadis, Brady and Watt in 
\cite{athanasiadis07shellability}.

\begin{theorem}[\cite{athanasiadis07shellability}*{Theorem~1.1}]\label{thm:shellability_real}
	The lattice $\nc{W}$ is EL-shellable for every real reflection group $W$.
\end{theorem}

If we concatenate Theorems~\ref{thm:shellability_main} and \ref{thm:shellability_real}, and exploit a fact about
the braid groups of the groups $G(d,1,n)$, for $d,n\geq 3$, and $G_{25},G_{26},G_{32}$, then we obtain the following result.

\begin{theorem}\label{thm:shellability_all}
	The lattice $\nc{W}$ is EL-shellable for every well-generated complex reflection group $W$.
\end{theorem}

It turns out that the main obstacle for a uniform proof of Theorem~\ref{thm:shellability_all} is not so much the definition of
an edge-labeling, but more the definition of a suitable total order on the reflections of $W$. A natural edge-labeling of 
$\nc{W}$ follows almost instantly from the group structure of $W$ and the definition of the partial order on $\nc{W}$. The proofs
of Theorems~\ref{thm:shellability_main} and \ref{thm:shellability_real} use this labeling and mainly deal with the definition of a 
total order of the reflections below a fixed Coxeter element $\gamma\in W$, a so-called \alert{$\gamma$-compatible reflection order}.
While in the case of real reflection groups such a reflection order can be defined uniformly, its existence remains case-by-case for 
the complex reflection groups. Moreover, in the case of real reflection groups, some instances of such an order benefit greatly 
from certain properties of the root systems associated with these groups, properties that cannot be generalized to complex 
reflection groups either. We refer the reader to Section~\ref{sec:uniform} or to \cite{muehle14on} for further information on 
$\gamma$-compatible reflection orders.

Another important aspect of Theorem~\ref{thm:shellability_all} is the connection of EL-shellability of $\nc{W}(\gamma)$ with the 
transitivity of the Hurwitz action on the set of reduced $T$-words of $\gamma$, see \cite{bessis14finite}*{Proposition~7.5}. It is 
an intriguing question whether we can derive the EL-shellability of $\nc{W}$ directly from the transitivity of this action or vice versa.

Once we have established Theorem~\ref{thm:shellability_all}, we use a construction of Armstrong and Thomas from
\cite{armstrong09generalized} to show the following, more general result.

\begin{theorem}\label{thm:generalized_shellability_all}
	Let $m\in\mathbb{N}$, and let $W$ be a well-generated complex reflection group. Denote by $\gnc{W}{m}$ the 
	poset of $m$-divisible $W$-noncrossing partitions, and let $\bnc{W}{m}$ be the lattice that 
	arises from $\gnc{W}{m}$ by adjoining a least element. Then $\bnc{W}{m}$ is EL-shellable.
\end{theorem}

\smallskip

This article is organized as follows. In Section~\ref{sec:preliminaries}, we give background information on complex
reflection groups, noncrossing partitions and EL-shellability. In Section~\ref{sec:gddn} we recall some basic facts
about the complex reflection groups $G(d,d,n)$ for $d,n\geq 3$, and prove the EL-shellability of $\nc{G(d,d,n)}$.
For the exceptional well-generated complex reflection groups, we explicitly construct an EL-labeling of the corresponding 
lattice of noncrossing partitions with the help of a computer program in Section~\ref{sec:shellability_exceptional}, and
thus conclude the proofs of Theorems~\ref{thm:shellability_main} and \ref{thm:shellability_all}. In 
Section~\ref{sec:shellability_generalized}, we briefly recall the construction of the EL-labeling of $\bnc{W}{m}$
given by Armstrong and Thomas in \cite{armstrong09generalized}*{Section~3.7} for the case where $W$ is a real 
reflection group, and conclude the proof of Theorem~\ref{thm:generalized_shellability_all}. Eventually, in 
Section~\ref{sec:application}, we present some applications of Theorem~\ref{thm:generalized_shellability_all}
concerning the M{\"o}bius function of $\gnc{W}{m}$ that were previously conjectured in
\cite{armstrong09euler} and \cite{tomie09mobius}. 

\section{Preliminaries}
  \label{sec:preliminaries}
In this section, we provide definitions and background for the concepts treated in this article. For a more detailed
introduction to complex reflection groups, we refer to \cite{lehrer09unitary}. EL-shellability of partially ordered
sets was introduced in \cite{bjorner80shellable}. More details and examples can be found there.

\subsection{Complex Reflection Groups}
  \label{sec:reflection_groups}
Let $V$ be an $n$-dimensional complex vector space and let $w\in U(V)$ be a unitary transformation on $V$. Define the
\alert{fixed space $\fix{w}$ of $w$} as the set of all vectors in $V$ that remain invariant under the action of $w$.
A unitary transformation $w$ is called a \alert{reflection} if it has finite order and the corresponding fixed space 
has codimension $1$. Hence $\fix{w}$ is a hyperplane in $V$, the so-called \alert{reflection hyperplane of $w$}. 
A finite subgroup $W\leq U(V)$ that is generated by unitary reflections is called a unitary reflection group or---
as we say throughout the rest of the article---a \alert{complex reflection group}. A complex reflection group is called
\alert{irreducible} if it does not fix a proper subspace of $V$. The \alert{rank of $W$} is the dimension of the complement 
of the fixed space $V^{W}$ in $V$. We say that $W$ is \alert{well-generated} if $W$ has rank $n$ and can be generated 
by $n$ reflections. 

According to Shephard and Todd's classification of finite irreducible complex reflection groups, see
\cite{shephard54finite}, there is one infinite family of such reflection groups, denoted by $G(d,e,n)$ with $d,e,n$
being positive integers such that $n$ is the rank of the group, and such that
$e$ divides $d$, as well as $34$ exceptional groups, denoted by $G_{4},G_{5},\ldots,G_{37}$. 
We call a square matrix having exactly one non-zero entry in each row and in each column a \alert{monomial matrix}. 
The group $G(d,e,n)$ admits a representation as a group of monomial $(n\times n)$-matrices in which each non-zero entry is 
a $d$-th root of unity and the product of all non-zero entries is a $\tfrac{d}{e}$-th root of unity, and we will 
refer to this representation as the \alert{standard monomial representation of $G(d,e,n)$}. See 
\cite{lehrer09unitary}*{Chapter~2.2} for the exact definition.

We can conclude from \cite{orlik80unitary}*{Table~2} that there are three infinite 
families of irreducible well-generated complex reflection groups, namely $G(1,1,n)$ for $n\geq 1$, $G(d,1,n)$ for
$n\geq 1,d\geq 2$, and $G(d,d,n)$ for $n,d\geq 2$. Among the $34$ exceptional irreducible complex reflection groups,
$26$ are well-generated, and we list them in Section~\ref{sec:shellability_exceptional}.

For every complex reflection group $W$ of rank $n$ there is a set of algebraically independent polynomials
$\sigma_{1},\sigma_{2},\ldots,\sigma_{n}\in\mathbb{C}[X_{1},X_{2},\ldots,X_{n}]$ that remain invariant under the 
action of $W$. The degrees of these polynomials are called the \alert{degrees of $W$}, and they are independent
of the actual choice of invariants, see \cite{chevalley55invariants}. They have a close connection to the structure of $W$. 
Namely, the product of the degrees equals the group order and their sum equals the number of reflections of $W$ plus $n$, see 
\cite{lehrer09unitary}*{Theorem~4.14}. 

\subsection{Regular Elements and Noncrossing Partitions}
  \label{sec:noncrossing_partitions}
In this section, we define the objects of interest for this article. Let $T=\{t_{1},t_{2},\ldots,t_{N}\}$ be the
set of all reflections of $W$. Since $W$ is generated by $T$, we can write every $w\in W$ as a product of reflections. 
This gives rise to a length function $\ell_{T}$ that assigns to every $w\in W$ the least number of reflections that 
are needed to form $w$. More formally, 
\begin{equation}\label{eq:length_function}
	\ell_{T}:W\to\mathbb{N},\quad 
	 w\mapsto\min\{k\mid w=t_{i_{1}}t_{i_{2}}\cdots t_{i_{k}},\;\text{where}\;1\leq i_{1},i_{2},\ldots,i_{k}\leq N\}.
\end{equation}
If $\ell_{T}(w)=k$, then we call every product of $k$ reflections that yields $w$ a \alert{reduced $T$-word of $w$}.
The reduced $T$-words of $w\in W$ can be transformed into one another as follows.

\begin{lemma}\label{lem:shifting}
	Let $W$ be a complex reflection group, and let $w\in W$ with a given reduced $T$-word $w=t_{1}t_{2}\cdots t_{k}$. For 
	every $i\in\{1,2,\ldots,k-1\}$ we have 
	\begin{displaymath}	
		w = t_{1}t_{2}\cdots t_{i-1}t_{i+1}(t_{i+1}^{-1}t_{i}t_{i+1})t_{i+2}\cdots t_{k} 
	\end{displaymath}
	is again a reduced $T$-word of $w$. Analogously for every $i\in\{2,3,\ldots,k\}$ we have 
	\begin{displaymath}
		w = t_{1}t_{2}\cdots t_{i-1}(t_{i}t_{i+1}t_{i}^{-1})t_{i}t_{i+2}\cdots t_{k}
	\end{displaymath}
	is again a reduced $T$-word of $w$.
\end{lemma}
\begin{proof}
	This follows immediately from the fact that conjugating a reflection yields a reflection again, see
	\cite{lehrer09unitary}*{Lemma~1.9}.
\end{proof}

With the help of the previously introduced length function, we can now define a partial order on $W$, the so-called 
\alert{absolute order of $W$}, by 
\begin{equation}\label{eq:absolute_order}
	u\leq_{T}v\quad\text{if and only if}\quad\ell_{T}(w)=\ell_{T}(u)+\ell_{T}(u^{-1}v).
\end{equation}
However, we are not interested in the whole poset $(W,\leq_{T})$, but in certain intervals thereof. To determine these
intervals, we need some more notation. Denote by $V$ the complex vector space on which $W$ acts. A vector 
$\mathbf{v}\in V$ is called \alert{regular} if it does not lie in any of the reflecting hyperplanes of $W$. If the eigenspace
to an eigenvalue $\zeta$ of $w\in W$ contains a regular vector, then $w$ is called \alert{$\zeta$-regular}, and 
the multiplicative order $d$ of $\zeta$ is called a \alert{regular number for $W$}. It follows from 
\cite{springer74regular}*{Theorem~4.2} that $\zeta$-regular elements form a single conjugacy class in $W$. 

It follows from \cite{lehrer99reflection}*{Theorem~C} that the largest degree, $d_{n}$, is a regular number for every well-generated complex
reflection group. 
If $\zeta$ is a $d_{n}$-th root of unity, then we call a $\zeta$-regular element of order $d_{n}$ a 
\alert{Coxeter element of $W$}, see \cite{reiner14on}*{Definition~1.1}, provided that $W$ is well-generated. The largest degree of 
$W$ is then called the \alert{Coxeter number of $W$}, and we usually write $h$ instead of $d_{n}$.

Now, let $W$ be an irreducible well-generated complex reflection group, let $\varepsilon\in W$ denote the identity of $W$, 
and let $\gamma\in W$ be a Coxeter element. We call the interval $[\varepsilon,\gamma]$ of $(W,\leq_{T})$ the 
\alert{lattice of $W$-noncrossing partitions}, and we denote it by $\nc{W}(\gamma)$. The following statement 
implies that $\nc{W}(\gamma)$ does not depend on the choice of $\gamma$.

\begin{proposition}[\cite{reiner14on}*{Corollary~1.6}]\label{prop:nc_independent}
	Let $W$ be an irreducible, well-generated complex reflection group, and let $\gamma,\gamma'\in W$ be two Coxeter elements.
	Then the posets $\nc{W}(\gamma)$ and $\nc{W}(\gamma')$ are isomorphic.
\end{proposition}

In view of Proposition~\ref{prop:nc_independent} we will suppress the chosen Coxeter element when it is not 
necessary, and write only $\nc{W}$. 

The fact that $\nc{W}$ is indeed a lattice for every irreducible well-generated complex reflection group was 
shown (case-by-case) in a series of papers, see 
\cites{bessis03dual,bessis06non,bessis14finite,brady01partial,brady02artin,brady08lattices}. (In \cite{brady01partial} it was 
shown that $\nc{G(1,1,n)}$ is a lattice, in \cite{brady02artin} it was shown that $\nc{G(2,1,n)}$ and $\nc{G(2,2,n)}$ are
lattices. In \cite{bessis03dual} the noncrossing partition lattices of all real reflection groups were considered. This
construction was extended to the groups $G(d,d,n)$ in \cite{bessis06non}, and later, in \cite{bessis14finite}, to all 
well-generated complex reflection groups. \cite{brady08lattices} provides
a uniform proof of the lattice property of $\nc{W}$ in the case where $W$ is a real reflection group, while a 
uniform proof of the lattice property of $\nc{W}$, where $W$ is a complex reflection group, has not yet appeared.)
It was also shown that $\nc{W}$ enjoys a number of beautiful properties: it is for instance graded, atomic, self-dual, 
locally self-dual, and complemented.

In \cite{armstrong09generalized}, Armstrong introduced a more general poset that he called the 
\alert{poset of $m$-divisible $W$-noncrossing partitions}, for some positive integer $m$. Given a 
Coxeter element $\gamma\in W$, this poset is defined by
\begin{multline*}
	\gnc{W}{m}(\gamma)=\left\{\vphantom{\sum_{i=0}^{m}}(w_{0};w_{1},\ldots,w_{m})\in\nc{W}^{m+1}\right\rvert\\
		\left.\gamma=w_{0}w_{1}\cdots w_{m}\;\text{and}\;\sum_{i=0}^{m}{\ell_{T}(w_{i})}=\ell_{T}(\gamma)\right\},
\end{multline*}
where the corresponding partial order is defined by
\begin{displaymath}
	(u_{0};u_{1},\ldots,u_{m})\leq(v_{0};v_{1},\ldots,v_{m})\quad\text{if and only if}\quad
	  v_{i}\leq_{T}u_{i}\;\text{for all}\;1\leq i\leq m.
\end{displaymath}
It turns out that $\Bigl(\gnc{W}{m}(\gamma),\leq\Bigr)$ is graded with rank function 
$\text{rk}(w_{0};w_{1},\ldots,w_{m})=\ell_{T}(w_{0})$, and has a greatest element 
$(\gamma;\varepsilon,\ldots,\varepsilon)$. In general, however, this poset has no least element. Again 
Proposition~\ref{prop:nc_independent} implies that $\gnc{W}{m}(\gamma)$ does not depend on the choice of $\gamma$, and we 
will thus drop the Coxeter element from the notation whenever it is not necessary.

Although Armstrong considered only real reflection groups, the same construction can be carried out in the general setting of
well-generated complex reflection groups, see \cite{bessis11cyclic}. Not surprisingly, the case $m=1$ yields the 
noncrossing partition lattice $\nc{W}(\gamma)$ as defined in the previous paragraph. By theorems of several authors
\cites{athanasiadis04noncrossing,bessis14finite,bessis06non,chapoton04enumerative,edelman80chain,reiner97non} it
follows that for every irreducible well-generated complex reflection group $W$ and every positive integer $m$, we have
\begin{equation}\label{eq:fuss_catalan}
	\Bigl\lvert\gnc{W}{m}\Bigr\rvert=\prod_{i=1}^{n}{\frac{mh+d_{i}}{d_{i}}},
\end{equation}
where the numbers $d_{i}$ denote the degrees of $W$ in nondecreasing order, and $h$ is the Coxeter number of $W$. 
The numbers appearing in \eqref{eq:fuss_catalan} are called \alert{Fu{\ss}-Catalan numbers of $W$}, and will be denoted by 
$\cat{W}{m}$.

\subsection{EL-Shellability of Graded Posets}
  \label{sec:shellability}
Let $(P,\leq_{P})$ be a finite graded poset. We call $(P,\leq_{P})$ \alert{bounded} if it has a least and a greatest
element. A chain $\mathbf{c}:x=p_{0}<_{P}p_{1}<_{P}\cdots<_{P}p_{k}=y$ in some interval $[x,y]$ of $(P,\leq_{P})$ 
is \alert{maximal} if there are no $q\in P$ and no $i\in\{0,1,\ldots,k-1\}$ such that $p_{i}<_{P}q<_{P}p_{i+1}$.
Denote by $\mathcal{E}(P)$ the set of edges in the Hasse diagram of $(P,\leq_{P})$. Given another poset 
$(\Lambda,\leq_{\Lambda})$, a map $\lambda:\mathcal{E}(P)\to\Lambda$ is called an 
\alert{edge-labeling of $(P,\leq_{P})$}. Let 
$\lambda(\mathbf{c})=\bigl(\lambda(p_{0},p_{1}),\lambda(p_{1},p_{2}),\ldots,\lambda(p_{k-1},p_{k})\bigr)$ denote the 
sequence of edge labels of $\mathbf{c}$ with respect to $\lambda$. We write $\lambda\bigl([x,y]\bigr)$ for the set
of sequences $\lambda(\mathbf{c})$, where $\mathbf{c}$ runs over all maximal chains in $[x,y]$. 

A maximal chain $\mathbf{c}$ in $[x,y]$ is 
\alert{increasing} if $\lambda(\mathbf{c})$ is a strictly increasing sequence. For another maximal chain 
$\mathbf{c'}:x=q_{0}<_{P}q_{1}<_{P}\cdots<_{P}q_{k}=y$ in the same interval, we say that 
\alert{$\mathbf{c}$ is lexicographically smaller than $\mathbf{c'}$} if $\lambda(\mathbf{c})$ is smaller than 
$\lambda(\mathbf{c'})$ with respect to the lexicographic order on $\Lambda^{k}$. If $\lambda$ is an edge-labeling
such that for every interval of $(P,\leq_{P})$ there exists a unique increasing maximal chain which is lexicographically
first among all maximal chains in this interval, then we call $\lambda$ an \alert{EL-labeling of $(P,\leq_{P})$}. 
A bounded, graded poset that admits an EL-labeling is called \alert{EL-shellable}. Recall the following result.

\begin{theorem}[\cite{bjorner80shellable}*{Theorem~4.3}]\label{thm:direct_product_shellable}
	Let $P,Q$ be bounded, graded posets. The direct product $P\times Q$ is EL-shellable if and only if both $P$ and
	$Q$ are EL-shellable.
\end{theorem}

EL-shellability of posets was introduced by Bj{\"o}rner in \cite{bjorner80shellable} as an order-theoretic tool to 
prove a conjecture by Stanley stating that a certain class of lattices is Cohen-Macaulay. This was obtained by showing
that the desired class of lattices is EL-shellable and that EL-shellability implies shellability which in turn 
implies Cohen-Macaulayness. Subsequently, EL-shellability turned out to be a powerful tool to investigate the
topological properties of posets. It was for instance shown that the number of falling maximal chains in an interval
$[x,y]$ of an EL-shellable poset $(P,\leq_{P})$ (with respect to the EL-labeling) equals $\mu(x,y)$, where $\mu$
denotes the M{\"o}bius function of $(P,\leq_{P})$. Using this connection, it is also possible to determine the Euler
characteristic and thus the homotopy type of the order complex associated with $(P,\leq_{P})$. 

\section{EL-Shellability of $\nc{G(d,d,n)}$ for $d,n\geq 3$}
  \label{sec:gddn}
In this section we prove that $\nc{G(d,d,n)}$ is EL-shellable for $d,n\geq 3$. In order to do that we fix a particular 
Coxeter element $\gamma$, see \eqref{eq:coxeter_gddn}, and define a total order on the set of reflections below $\gamma$ 
in $\nc{G(d,d,n)}(\gamma)$, see \eqref{eq:order_gddn}.
Subsequently we show that the natural labeling of $\nc{G(d,d,n)}(\gamma)$, denoted by $\lambda_{\gamma}$ that assigns to 
each cover relation $(u,v)$ in $\nc{G(d,d,n)}(\gamma)$ the unique reflection $u^{-1}v$ is an EL-labeling of 
$\nc{G(d,d,n)}(\gamma)$, see Theorem~\ref{thm:shellability_gddn} below. The proof proceeds by induction on $n$. First we 
show that $\lambda_{\gamma}$ is an EL-labeling for all possible intervals of length $2$, see Lemma~\ref{lem:gddn_rank_2} 
below. Subsequently we show the same for all possible intervals of length $n-1$ in $\nc{G(d,d,n)}(\gamma)$, see 
Propositions~\ref{prop:gddn_symmetric_irreducible} and \ref{prop:gddn_gend_irreducible} as well as
Corollaries~\ref{cor:gddn_symmetric_reducible} and \ref{cor:gddn_gend_reducible}. The proofs of these intermediate 
steps are quite technical, and we thus only present an outline in the text, and refer the reader to the appendix for all the 
details. Finally, we use these results to establish that $\lambda_{\gamma}$ is an EL-labeling of $\nc{G(d,d,n)}(\gamma)$.

\subsection{The Setup}
  \label{sec:gddn_setup}
First of all, we recall that in the standard monomial representation of $G(d,d,n)$, the elements of $G(d,d,n)$ are
monomial matrices whose nonzero entries are $d$-th roots of unity, and the product of all nonzero entries is $1$.
Thus $G(d,d,n)$ can be seen as a subgroup of the symmetric group $\mathfrak{S}_{dn}$, acting on the set
\begin{displaymath}
	\Bigl\{\colint{1}{0},\colint{2}{0},\ldots,\colint{n}{0},\colint{1}{1},\colint{2}{1},\ldots,\colint{n}{1},\ldots,
	  \colint{1}{d-1},\colint{2}{d-1},\ldots,\colint{n}{d-1}\Bigr\}
\end{displaymath}
of integers with $d$ colors such that $w\in G(d,d,n)$ satisfies 
\begin{displaymath}
	w\bigl(\colint{k}{s}\bigr)=\colint{\pi(k)}{s+t_{k}}\quad\text{and}\quad\sum_{i=1}^{k}{t_{k}}\equiv 0\pmod{d},
\end{displaymath}
where $\pi\in\mathfrak{S}_{n}$, and the numbers $t_{k}$ depend only on $w$ and $k$. (The addition in the
superscript is understood modulo $d$.) Thus the elements of $G(d,d,n)$ can be decomposed into cycles of the 
following form:
\begin{multline*}
	\bigl(\!\bigl(\colint{k_{1}}{t_{1}}\;\colint{k_{2}}{t_{2}}\;\ldots\;\colint{k_{r}}{t_{r}}\bigr)\!\bigr) = 
	  \bigl(\colint{k_{1}}{t_{1}}\;\colint{k_{2}}{t_{2}}\;\ldots\;\colint{k_{r}}{t_{r}}\bigr)\\
	  \bigl(\colint{k_{1}}{t_{1}+1}\;\colint{k_{2}}{t_{2}+1}\;\ldots\;\colint{k_{r}}{t_{r}+1}\bigr)\cdots
	  \bigl(\colint{k_{1}}{t_{1}+d-1}\;\colint{k_{2}}{t_{2+d-1}}\;\ldots\;\colint{k_{r}}{t_{r}+d-1}\bigr)
\end{multline*}
and 
\begin{multline*}
	\bigl[\colint{k_{1}}{t_{1}}\;\colint{k_{2}}{t_{2}}\;\ldots\;\colint{k_{r}}{t_{r}}\bigr]_{s} =
	  \bigl(\colint{k_{1}}{t_{1}}\;\colint{k_{2}}{t_{2}}\;\ldots\;\colint{k_{r}}{t_{r}}\\
	  \colint{k_{1}}{t_{1}+s}\;\colint{k_{2}}{t_{2}+s}\;\ldots\;\colint{k_{r}}{t_{r}+s}\;
	  \ldots\;\colint{k_{1}}{t_{1}+(d-1)s}\;\colint{k_{2}}{t_{2}+(d-1)s}\;\ldots\;\colint{k_{r}}{t_{r}+(d-1)s}\bigr),
\end{multline*}
Usually, we will simply write $\bigl[\colint{k_{1}}{t_{1}}\;\colint{k_{2}}{t_{2}}\;\ldots\;\colint{k_{r}}{t_{r}}\bigr]$ 
instead of $\bigl[\colint{k_{1}}{t_{1}}\;\colint{k_{2}}{t_{2}}\;\ldots\;\colint{k_{r}}{t_{r}}\bigr]_{1}$.

\subsection{Parabolic Subgroups}
  \label{sec:gddn_parabolic}
Let $V$ denote the complex vector space on which $G(d,d,n)$ acts. We call the maximal subgroup of $G(d,d,n)$ that 
fixes some $A\subseteq V$ pointwise a \alert{parabolic subgroup of $G(d,d,n)$}. 

\begin{lemma}\label{lem:parabolic_subgroups}
	Let $W$ be a parabolic subgroup of $G(d,d,n)$. If $W$ is irreducible, then $W$ is either isomorphic to
	$G(1,1,n')$ or to $G(d,d,n')$ for $n'\leq n$. If $W$ is reducible, then $W$ is isomorphic to a direct product
	of irreducible parabolic subgroups of $G(d,d,n)$. 
\end{lemma}
\begin{proof}
	This follows from \cite{broue98complex}*{Fact~1.7} and \cite{broue98complex}*{Table~2}.
\end{proof}

The following property of Coxeter elements in well-generated complex reflection groups was observed by Ripoll.
\begin{proposition}[\cite{ripoll10orbites}*{Proposition~6.3(i),(ii)}]\label{prop:parabolic_coxeter_elements}
	Let $W$ be a well-generated complex reflection group, and let $w\in W$. Let $T$ denote the set of all reflections
	of $W$. The following are equivalent:
	\begin{enumerate}[(i)]
		\item $w$ is a Coxeter element in a parabolic subgroup of $W$, \quad and
		\item there is a Coxeter element $\gamma_{w}\in W$ such that $w\leq_{T}\gamma_{w}$.
	\end{enumerate}
\end{proposition}

We call $w$ a \alert{parabolic Coxeter element} if it satisfies one of the properties stated in 
Proposition~\ref{prop:parabolic_coxeter_elements}. Analogously to real reflection groups, the length of a 
parabolic Coxeter element of $G(d,d,n)$ is determined by the codimension of its fixed space. 

\begin{lemma}[\cite{bessis06non}*{Lemma~4.1(ii)}]\label{lem:gddn_length_fixed_space}
	For $w\in\nc{G(d,d,n)}$, we have $\ell_{T}(w)=n-\dim\fix{w}$.
\end{lemma}

\subsection{Reflections and Coxeter Element}
  \label{sec:gddn_reflections_coxeter}
One of the major differences between real and complex reflection groups is the fact that real reflections are involutions, 
while complex reflections may have order $>2$. The following proposition shows that $G(d,d,n)$ is well-behaved with 
respect to this aspect. 

\begin{proposition}[\cite{lehrer09unitary}*{Proposition~2.9}]\label{prop:reflections_gddn}
	The group $G(d,d,n)$ contains $d\tbinom{n}{2}$ reflections and the order of every reflection is two.
\end{proposition}

Let us have a closer look at the standard monomial representation of the reflections of $G(d,d,n)$: since they are unitary 
involutions that fix a space of codimension $1$, it follows immediately that we have
\begin{equation}\label{eq:gddn_reflections}
	T = \Bigl\{\colref{a}{b}{s}\mid 1\leq a<b\leq n,0\leq s<d\Bigr\}.
\end{equation}	
Now let us emphasize a certain subset of $T$, namely the reflections
\begin{displaymath}
	\colref{1}{2}{0},\colref{2}{3}{0},\ldots,\colref{(n\!-\!1)}{n}{0},\colref{(n\!-\!1)}{n}{1},
\end{displaymath}
which we call the \alert{simple reflections of $G(d,d,n)$}, and which we abbreviate by 
$s_{i}=\colref{i}{(i\!+\!1)}{0}$ for $1\leq i<n$ and $s_{n}=\colref{(n\!-\!1)}{n}{1}$. Their product 
$\gamma=s_{1}s_{2}\cdots s_{n}$ is the group element
\begin{equation}\label{eq:coxeter_gddn}
	\gamma=\bigl[\colint{1}{0}\;\colint{2}{0}\;\ldots\;\colint{(n\!-\!1)}{0}\bigr]\bigl[\colint{n}{0}\bigr]_{d-1},
\end{equation}
which can be represented by the monomial matrix
\begin{equation}\label{eq:coxeter_gddn_matrix}
	C=\left(\begin{array}{ccccccc}
	  0 & 0 & 0 & \cdots & 0 & \zeta_{d} & 0\\
	  1 & 0 & 0 & \cdots & 0 & 0 & 0 \\
	  0 & 1 & 0 & \cdots & 0 & 0 & 0 \\
	  \vdots & \vdots & \vdots & & \vdots & \vdots & \vdots\\
	  0 & 0 & 0 & \cdots & 1 & 0 & 0 \\
	  0 & 0 & 0 & \cdots & 0 & 0 & \zeta_{d}^{d-1}\\
	\end{array}\right),
\end{equation}
where $\zeta_{d}=e^{2\pi\sqrt{-1}/d}$ is a $d$-th root of unity. Recall for instance from 
\cite{orlik80unitary}*{Table~2} that the degrees of $G(d,d,n)$ are
\begin{equation}\label{eq:degrees_gddn}
	d,2d,\ldots,(n-1)d,n, 
\end{equation}
and hence that the Coxeter number of $G(d,d,n)$ is $h=(n-1)d$. We can check that $\zeta_{h}$ is an eigenvalue
of $C$, and an eigenvector of $C$ to $\zeta_{h}$ is for instance
\begin{equation}\label{eq:regular_vector}
	\vb=\left(\begin{array}{ccccc}\zeta_{h}^{n-1} & \zeta_{h}^{n-2} & \ldots & \zeta_{h} & 0\end{array}\right)^{\tp},
\end{equation}
where ``$\tp$'' denotes the transposition of vectors. The reflection hyperplanes of $G(d,d,n)$ (in standard monomial
representation) are given by the equations
\begin{displaymath}
	x_{i}=\zeta_{d}^{s}x_{j},\quad\mbox{for}\;1\leq i<j\leq n\;\text{and}\;0\leq s<d.
\end{displaymath}
Hence the vector $\vb$ from \eqref{eq:regular_vector} is indeed $\zeta_{h}$-regular, which makes $\gamma$ a 
Coxeter element of $G(d,d,n)$. For later use, we refer to the reduced $T$-word
\begin{equation}\label{eq:factors_gddn}
	\gamma=\colref{1}{2}{0}\colref{2}{3}{0}\cdots\colref{(n\!-\!1)}{n}{0}\colref{(n\!-\!1)}{n}{1},
\end{equation}
as the \alert{simple decomposition of $\gamma$}. From now on, whenever we write $\nc{G(d,d,n)}$ we actually mean 
$\nc{G(d,d,n)}(\gamma)$, \ie the interval $[\varepsilon,\gamma]$ in $\bigl(G(d,d,n),\leq_{T}\bigr)$ for the Coxeter
element $\gamma$ from \eqref{eq:coxeter_gddn}.

\begin{remark}\label{rem:sym_subgroup}
	If we consider the subword 
	$\bar{\gamma}=\gamma s_{n}=\bigl(\!\bigl(\colint{1}{0}\;\colint{2}{0}\;\ldots\;\colint{n}{0}\bigr)\!\bigr)$, then we obtain 
	a reduced $T$-word
	\begin{equation}\label{eq:factors_g11n}
		\bar{\gamma}=\colref{1}{2}{0}\colref{2}{3}{0}\cdots\colref{(n\!-\!1)}{n}{0},
	\end{equation}
	which we will refer to as the \alert{simple decomposition of $\bar{\gamma}$}. More precisely, it can be checked that
	$\bar{\gamma}$ is a Coxeter element in the parabolic subgroup $G(1,1,n)$ (which has rank $n-1$) of $G(d,d,n)$, and thus
	we call the reflections $s_{1},s_{2},\ldots,s_{n-1}$ the \alert{simple reflections of $G(1,1,n)$}. Indeed, there is an 
	obvious bijection between the set $\{s_{1},s_{2},\ldots,s_{n-1}\}$ and the set of transpositions 
	$\bigl\{(1\;2),(2\;3),\ldots,(n\!-\!1\;n)\bigr\}$ which forms a set of canonical generators for the symmetric group
	$\mathfrak{S}_{n}$.
\end{remark}

Another difference between real and complex reflection groups is that in a complex reflection group, not all 
reflections have to be comparable (with respect to the absolute order) to a given Coxeter element\footnote{I thank an anonymous
referee for pointing out to me that there are also cases of infinite real reflection groups in which not all reflections are 
comparable to a given Coxeter element. See for instance \cite{mccammond13dual}*{Theorem~9.6}.}. Let us therefore define 
$T_{\gamma}=\{t\in T\mid t\leq_{T}\gamma\}$. 

\begin{proposition}
  \label{prop:gamma_reflections}
	Let $\gamma$ be the Coxeter element of $G(d,d,n)$ as defined in \eqref{eq:coxeter_gddn}. Then we have
	\begin{align*}
		T_{\gamma} & = \Bigl\{\colref{a}{b}{s}\mid 1\leq a<b<n,s\in\{0,d-1\}\Bigr\}\\
			& \kern1cm\cup\Bigl\{\colref{a}{n}{s}\mid 1\leq a<n,0\leq s\leq d-1\Bigr\}.
	\end{align*}
\end{proposition}
\begin{proof}
	Proposition~\ref{prop:reflections_gddn} implies that the reflections of $G(d,d,n)$ are involutions, and by 
	definition we have $\ell_{T}(t)=1$ if and only if $t\in T$. Hence it follows from 
	Lemma~\ref{lem:gddn_length_fixed_space} that $t\leq_{T}\gamma$ if and only if $\dim\fix{t\gamma}=1$.

	Let $\vb=(v_{1},v_{2},\ldots,v_{n})^{\tp}\in\mathbb{C}^{n}$ be an arbitrary vector. Then we have
	\begin{equation}\label{eq:coxeter_action}
		\vb'=\gamma\vb
		  = \left(\zeta_{d}v_{n-1},v_{1},v_{2},\ldots,v_{n-2},\zeta_{d}^{d-1}v_{n}\right)^{\tp}.
	\end{equation}
	In what follows, we determine the dimension of $\fix{t\gamma}$ for $t\in T$. Recall that for 
	$w\in G(d,d,n)$, the fixed space of $w$ is defined as 
	$\fix{w}=\{\vb\in\mathbb{C}^{n}\mid w\vb=\vb\}$. We distinguish three cases:
	
	(i) If $t=\colref{a}{b}{s}$, where $1\leq a<b<n$ and $0\leq s<d$, then we obtain
	\begin{align*}
		t\vb'
		  & = t\left(\zeta_{d}v_{n-1},v_{1},v_{2},\ldots,v_{n-2},\zeta_{d}^{d-1}v_{n}\right)^{\tp}\\
		  & = \left(\zeta_{d}v_{n-1},v_{1},\ldots,v_{a-2},\zeta_{d}^{d-s}v_{b-1},v_{a},\ldots,v_{b-2},
		  \zeta_{d}^{s}v_{a-1},v_{b},\ldots,v_{n-2},\zeta_{d}^{d-1}v_{n}\right)^{\tp}.
	\end{align*}
	Thus $\fix{t\gamma}$ is given by the following system of linear equations:
	\begin{displaymath}\begin{aligned}
		& v_{1}=\zeta_{d}v_{n-1}, && v_{2}=v_{1}, && v_{3}=v_{2}, && \ldots, && v_{a-1}=v_{a-2},\\
		& v_{a}=\zeta_{d}^{d-s}v_{b-1}, && v_{a+1}=v_{a}, && v_{a+2}=v_{a+1}, && \ldots, && v_{b-1}=v_{b-2},\\
		& v_{b}=\zeta_{d}^{s}v_{a-1}, && v_{b+1}=v_{b}, && v_{b+2}=v_{b+1}, && \ldots, && v_{n-1}=v_{n-2},\\
		& v_{n}=\zeta_{d}^{d-1}v_{n}.
	\end{aligned}\end{displaymath}
	If we put these equations together, then we obtain
	\begin{align*}
		\zeta_{d}^{s+1}v_{n-1} & = \zeta_{d}^{s}v_{1} = \cdots = \zeta_{d}^{s}v_{a-1} = v_{b} = \cdots = v_{n-1},\\
		\zeta_{d}^{d-s}v_{b-1} & = v_{a} = \cdots = v_{b-1},\\
		\zeta_{d}^{d-1}v_{n} & = v_{n}.
	\end{align*}
	The first line has a nontrivial solution only if $s=d-1$ (which forces the components in lines $2$ and $3$ to 
	be zero), and hence $\dim\fix{t\gamma}=1$. Similarly, the second line has a nontrivial solution only if $s=0$
	(which forces the components in lines $1$ and $3$ to be zero), and hence $\dim\fix{t\gamma}=1$. Thus in these
	two cases, we obtain $t\leq_{T}\gamma$. Every other value of $s$ forces all components to be zero, and hence
	$\dim\fix{t\gamma}=0$, which implies that $t\not\leq_{T}\gamma$.
	
	(ii) If $t=\colref{1}{n}{s}$, where $0\leq s<d$, then we obtain
	\begin{displaymath}
		t\vb' = \left(\zeta_{d}^{d-s-1}v_{n},v_{1},\ldots,v_{n-2},\zeta_{d}^{s+1}v_{n-1}\right)^{\tp}.
	\end{displaymath}
	Analogously to (i), we see that $\dim\fix{t\gamma}=1$, which implies $t\leq_{T}\gamma$.
	
	(iii) If $t=\colref{a}{n}{s}$, where $1<a<n$ and $0\leq s<d$, then we obtain
	\begin{displaymath}
		t\vb' = \left(\zeta_{d}v_{n-1},v_{1},\ldots,v_{a-2},\zeta_{d}^{d-s-1}v_{n},v_{a},\ldots,
		    v_{n-2},\zeta_{d}^{s}v_{a-1}\right)^{\tp}.
	\end{displaymath}
	Again, analogously to (i), we see that $\dim\fix{t\gamma}=1$, which implies $t\leq_{T}\gamma$.
\end{proof}

\subsection{The Labeling}
  \label{sec:gddn_labeling}
In order to prove the EL-shellablity of $\nc{G(d,d,n)}$, we need to find a suitable EL-labeling. A good candidate for
such a labeling arises quite naturally from the group structure of $G(d,d,n)$:
\begin{equation}\label{eq:labeling_gddn}
	\lambda_{\gamma}:\mathcal{E}\bigr(\nc{G(d,d,n)}\bigr)\to T_{\gamma},\quad (u,v)\mapsto u^{-1}v.
\end{equation}
The analogous labeling was already used in \cite{athanasiadis07shellability} to prove the EL-shellability of the noncrossing
partition lattices associated with real reflection groups. Let us first discuss some basic properties of this labeling.

\begin{lemma}\label{lem:reduced_word_maximal_chain}
	Let $u,v\in\nc{G(d,d,n)}$ with $u\leq_{T}v$. A product $t_{1}t_{2}\cdots t_{k}$ is a 
	reduced $T$-word of $u^{-1}v$ if and only if there exists a maximal chain
	$\mathbf{c}:u=x_{0}<_{T}x_{1}<_{T}\cdots<_{T}x_{k}=v$ in $\nc{G(d,d,n)}$ with 
	$\lambda_{\gamma}(\mathbf{c})=(t_{1},t_{2},\ldots,t_{k})$.
\end{lemma}
\begin{proof}
	Let $\mathbf{c}:u=x_{0}<_{T}x_{1}<_{T}\cdots<_{T}x_{k}=v$ be a maximal chain in $[u,v]$ with 
	$\lambda_{\gamma}(\mathbf{c})=(t_{1},t_{2},\ldots,t_{k})$. Since $\nc{G(d,d,n)}$ is graded, we conclude
	$\ell_{T}(u^{-1}v)=k$. By definition of $\lambda_{\gamma}$, we obtain $x_{i-1}^{-1}x_{i}=t_{i}$ for all 
	$i\in\{1,2,\ldots,k\}$. Thus
	\begin{displaymath}
		t_{1}t_{2}\cdots t_{k}=x_{0}^{-1}x_{1}x_{1}^{-1}x_{2}\cdots x_{k-1}^{-1}x_{k}=u^{-1}v,
	\end{displaymath}
	as desired.
	
	On the other hand, let $t_{1}t_{2}\cdots t_{k}$ be a reduced $T$-word of $u^{-1}v$. Define
	$x_{0}=u$ and $x_{i}=ut_{1}t_{2}\cdots t_{i}$ for $i\in\{1,2,\ldots,k\}$. Then we have $x_{k}=v$, 
	and $x_{i-1}^{-1}x_{i}=t_{i-1}t_{i-2}\cdots t_{1}u^{-1}ut_{1}t_{2}\cdots t_{i}=t_{i}$. This implies 
	$x_{i-1}<_{T}x_{i}$, and $\ell_{T}(x_{i})=\ell_{T}(x_{i-1})+1$ for all $i\in\{1,2,\ldots,k\}$. Since 
	$\nc{G(d,d,n)}$ is graded, we conclude that $\mathbf{c}:u=x_{0}<_{T}x_{1}<_{T}\cdots<_{T}x_{k}=v$ is a maximal
	chain in $[u,v]$ with $\lambda_{\gamma}(\mathbf{c})=(t_{1},t_{2},\ldots,t_{k})$.
\end{proof}

Because of the previous lemma we will use the expressions ``reduced $T$-word of $w$'' and 
``maximal chain in $[\varepsilon,w]$'' interchangeably. In particular, we call a reduced $T$-word of $w$ 
\alert{increasing} if the corresponding maximal chain in $[\varepsilon,w]$ is increasing with respect to 
$\lambda_{\gamma}$. 

\begin{lemma}\label{lem:chain_labels}
	Let $[u,v]$ be a non-singleton interval in $\nc{G(d,d,n)}$. 
	\begin{enumerate}[(i)]
		\item If $[u,v]$ has rank two and $(r,t)\in\lambda_{\gamma}\bigl([u,v]\bigr)$, then 
			$(t,r')\in\lambda_{\gamma}\bigl([u,v]\bigr)$ for some $r'\in T_{\gamma}$.
		\item If $t\in T_{\gamma}$ appears in some coordinate of an element $\lambda_{\gamma}\bigl([u,v]\bigr)$, 
			then $t=\lambda_{\gamma}(u,u')$ for some cover relation $(u,u')$ in $[u,v]$.
		\item The reflections appearing as the coordinates of an element of $\lambda_{\gamma}\bigl([u,v]\bigr)$ 
			are pairwise distinct.
	\end{enumerate}
\end{lemma}
\begin{proof}
	(i) Let $(r,t)\in\lambda_{\gamma}\bigl([u,v]\bigr)$. Recall from Proposition~\ref{prop:reflections_gddn} that 
	$t^{-1}=t$. Since $[u,v]$ has rank two, we conclude that $u^{-1}v=rt$. Now we apply 
	Lemma~\ref{lem:shifting} to obtain $u^{-1}v=t(trt)$ and with \cite{lehrer09unitary}*{Lemma~1.9} 
	and Lemma~\ref{lem:reduced_word_maximal_chain} follows that $r'=trt\in T_{\gamma}$. 
	
	(ii) This follows from repeated application of (i).
	
	(iii) Let $\mathbf{c}$ be a maximal chain in $[u,v]$ with 
	$\lambda_{\gamma}(\mathbf{c})=(t_{1},t_{2},\ldots,t_{k})$. It follows from Lemma~\ref{lem:reduced_word_maximal_chain} 
	that $\ell_{T}(u^{-1}v)=k$. Suppose that there exist indices $i<j$ with $t_{i}=t_{j}$. In view of (ii) we can find 
	a maximal chain $\mathbf{c'}$ in $[u,v]$ with
	$\lambda_{\gamma}(\mathbf{c'})=(t_{1},t_{2},\ldots,t_{i},t_{j},t'_{i+2},t'_{i+3},\ldots,t'_{k})$. 
	By assumption and Proposition~\ref{prop:reflections_gddn}, we obtain $t_{i}t_{j}=t_{i}^{2}=\varepsilon$, and 
	Lemma~\ref{lem:reduced_word_maximal_chain} implies now that 
	$(t_{1},t_{2},\ldots,t_{i-1},t'_{i+2},t'_{i+3},\ldots,t'_{k})$ is a reduced $T$-word of $u^{-1}v$ 
	which has length $k-2$. This, however, contradicts $\ell_{T}(u^{-1}v)=k$.
\end{proof}

\begin{lemma}\label{lem:interval_isomorphism}
	Let $[u,v]$ be a non-singleton interval in $\nc{G(d,d,n)}$. The poset isomorphism
	$f:[\varepsilon,u^{-1}v]\to[u,v]$ given by $f(x)=ux$ satisfies 
	$\lambda_{\gamma}(x,y)=\lambda_{\gamma}\bigl(f(x),f(y)\bigr)$ for all cover relations $(x,y)$ in 
	$[\varepsilon,u^{-1}v]$.
\end{lemma}
\begin{proof}
	This is straightforward to verify.
\end{proof}

Now we can prove the following proposition.

\begin{proposition}\label{prop:increasing_chain}
	Let $\gamma\in G(d,d,n)$ be the Coxeter element defined in \eqref{eq:coxeter_gddn}, and let $\lambda_{\gamma}$ be the 
	edge-labeling of $\nc{G(d,d,n)}$ defined in \eqref{eq:labeling_gddn}. For any total order of $T_{\gamma}$ 
	and any non-singleton interval $[u,v]$ in $\nc{G(d,d,n)}$, the lexicographically smallest maximal chain in 
	$[u,v]$ is increasing with respect to $\lambda_{\gamma}$.
\end{proposition}
\begin{proof}
	We follow the proof suggested for the analogous statement for real reflection groups in 
	\cite{athanasiadis07shellability}*{Theorem~3.5(i)}. 
	
	Let $[u,v]$ be a non-singleton interval of $\nc{G(d,d,n)}(\gamma)$, and let $\prec$ be a total order of 
	$T_{\gamma}$. We proceed by induction on $\ell_{T}(u^{-1}v)$. If $\ell_{T}(u^{-1}v)=1$, then the statement is 
	trivial. So, suppose that $\ell_{T}(u^{-1}v)=k$, and the statement is true for all intervals 
	$[u',v']$ in $\nc{G(d,d,n)}$ with $\ell_{T}({u'}^{-1}v')<k$. It is easy to see that 
	all cover relations $(u,\bar{u})$ with $\bar{u}\leq_{T}v$ have a different label with respect to 
	$\lambda_{\gamma}$. Now, let 
	$t=\min\bigl\{\lambda_{\gamma}(u,ut)\mid t\in T_{\gamma}\;\text{and}\;ut\leq_{T}v\bigr\}$, where the 
	minimum is taken with respect to $\prec$. Suppose that there is a chain in $[ut,v]$ having an edge labeled by 
	a reflection $r$ with $r\prec t$. Then Lemma~\ref{lem:chain_labels}(ii) implies that $(u,ur)$ is a cover
	in $[u,v]$, contradicting the choice of $t$. Moreover, Lemma~\ref{lem:chain_labels}(iii) implies that $t$ does 
	not occur in $\lambda\bigl([ut,v]\bigr)$. By induction assumption, the lexicographic smallest maximal chain in 
	$[ut,v]$ is increasing with respect to $\prec$. By the previous reasoning we can append this chain to the cover 
	$(u,ut)$, which implies that the lexicographic smallest maximal chain in $[u,v]$ is increasing with respect 
	to $\prec$.
\end{proof}

In fact, the previous proposition holds not only for the particular Coxeter element defined in \eqref{eq:coxeter_gddn}, but
in view of Proposition~\ref{prop:nc_independent} for any Coxeter element of $G(d,d,n)$.

\subsection{The Proof}
  \label{sec:gddn_proof}
In this section, we prove Theorem~\ref{thm:shellability_main} for the case where $W=G(d,d,n)$ with $d,n\geq 3$. We have
seen in Proposition~\ref{prop:increasing_chain} that the lexicographically smallest maximal chain in $\nc{G(d,d,n)}$ is increasing
for any total order of $T_{\gamma}$. We will next choose a particular total order on $T_{\gamma}$, and show that with
respect to this order, there exists a unique increasing maximal chain in every interval of $\nc{G(d,d,n)}$. This order 
is denoted by $\prec_{\gamma}$, and is given by:
\begin{equation}\label{eq:order_gddn}\begin{aligned}
	&& \colref{1}{2}{0} \prec_{\gamma} && \colref{1}{3}{0} \prec_{\gamma} && \cdots \prec_{\gamma} && \colref{1}{(n\!-\!1)}{0}\\
	&& \prec_{\gamma} && \colref{2}{3}{0} \prec_{\gamma} && \cdots \prec_{\gamma} && \colref{2}{(n\!-\!1)}{0}\\
	&& \prec_{\gamma} && \colref{3}{4}{0} \prec_{\gamma} && \cdots \prec_{\gamma} && \colref{(n\!-\!2)}{(n\!-\!1)}{0}\\
	\prec_{\gamma} && \colref{1}{n}{0} \prec_{\gamma} && \colref{1}{n}{d-1} \prec_{\gamma} && \cdots \prec_{\gamma} 
		&& \colref{1}{n}{1}\\
	&& \prec_{\gamma} && \colref{1}{2}{d-1} \prec_{\gamma} && \cdots \prec_{\gamma} && \colref{1}{(n\!-\!1)}{d-1}\\
	\prec_{\gamma} && \colref{2}{n}{0} \prec_{\gamma} && \colref{2}{n}{d-1} \prec_{\gamma} && \cdots \prec_{\gamma} 
		&& \colref{2}{n}{1}\\
	&& \prec_{\gamma} && \colref{2}{3}{d-1} \prec_{\gamma} && \cdots \prec_{\gamma} && \colref{2}{(n\!-\!1)}{d-1}\\
	\prec_{\gamma} && \colref{3}{n}{0} \prec_{\gamma} && \colref{3}{n}{d-1} \prec_{\gamma} && \cdots \prec_{\gamma} 
		&& \colref{(n\!-\!1)}{n}{1}.\\
\end{aligned}\end{equation}

Now we can state the main result of this section.

\begin{theorem}\label{thm:shellability_gddn}
	Let $d,n\geq 3$, let $\gamma\in G(d,d,n)$ be the Coxeter element defined in \eqref{eq:coxeter_gddn},
	let $T_{\gamma}$ be the set of reflections in $\nc{G(d,d,n)}(\gamma)$, and let $\lambda_{\gamma}$ be the 
	edge-labeling of $\nc{G(d,d,n)}(\gamma)$ defined in \eqref{eq:labeling_gddn}. If $T_{\gamma}$ is ordered
	as in \eqref{eq:order_gddn}, then $\lambda_{\gamma}$ is an EL-labeling of $\nc{G(d,d,n)}(\gamma)$.
\end{theorem}

The proof of this theorem consists of several steps which we present separately in the following statements.

\begin{lemma}\label{lem:sym_shellable}
	Let $\gamma$ be the Coxeter element of $G(d,d,n)$ defined in \eqref{eq:coxeter_gddn}. If $d=1$, then for every 
	$w\leq_{T}\gamma$ there exists a unique increasing reduced $T$-word of $w$ with respect to \eqref{eq:order_gddn}.
\end{lemma}
\begin{proof}
	The complex reflection group $G(1,1,n)$ is isomorphic to the symmetric group $\mathfrak{S}_{n}$, and under this 
	isomorphism, $\gamma$ corresponds to the long cycle $(1\;2\;\ldots\;n)$. 
	Then we have $T_{\gamma}=\Bigl\{\colref{a}{b}{0}\mid 1\leq a<b\leq n\Bigr\}$, and the total order $\prec_{\gamma}$ 
	restricts to that from \cite{athanasiadis07shellability}*{Example~3.3}, which is the lexicographic order on the 
	transpositions of $\mathfrak{S}_{n}$. Now \cite{athanasiadis07shellability}*{Theorem~3.5(ii)} implies the claim.
\end{proof}

The next lemma states what the coatoms of $\nc{G(d,d,n)}(\gamma)$ look like.

\begin{lemma}\label{lem:gamma_coatoms}
	Let $\gamma$ be the Coxeter element of $G(d,d,n)$ as defined in \eqref{eq:coxeter_gddn}, and let $t\in T_{\gamma}$. 
	If $t=\colref{a}{b}{0}$ for $1\leq a<b<n$, then we have 
	\begin{displaymath}
		\gamma t=\bigl[\colint{1}{0}\;\cdots\;\colint{a}{0}\;\colint{(b\!+\!1)}{0}\;\cdots\;\colint{(n\!-\!1)}{0}\bigr]
		\bigl[\colint{n}{0}\bigr]^{-1}\bigl(\!\bigl(\colint{(a\!+\!1)}{0}\;\cdots\;\colint{b}{0}\bigr)\!\bigr).
	\end{displaymath}
	If $t=\colref{a}{b}{d-1}$ for $1\leq a<b<n$, then we have
	\begin{displaymath}
		\gamma t=\bigl(\!\bigl(\colint{1}{0}\;\cdots\;\colint{a}{0}\;\colint{(b\!+\!1)}{d-1}\;\cdots\;\
			\colint{(n\!-\!1)}{d-1}\bigr)\!\bigr)\bigl[\colint{(a\!+\!1)}{0}\;\cdots\;\colint{b}{0}\bigr]
			\bigl[\colint{n}{0}\bigr]^{-1}.
	\end{displaymath}
	If $t=\colref{a}{n}{s}$ for $1\leq a<n$ and $0\leq s<d$, then we have
	\begin{displaymath}
		\gamma t=\bigl(\!\bigl(\colint{1}{0}\;\cdots\;\colint{a}{0}\;\colint{n}{s-1}\;\colint{(a\!+\!1)}{d-1}\;\cdots\;
			\colint{(n\!-\!1)}{d-1}\bigr)\!\bigr).
	\end{displaymath}
\end{lemma}
\begin{proof}
	This is a straightforward computation.
\end{proof}

Now we show that $\lambda_{\gamma}$ together with the total order of $T_{\gamma}$ from \eqref{eq:order_gddn} is an EL-labeling of 
the intervals $[\varepsilon,w]$ with $w\leq_{T}\gamma$ and $\ell_{T}(w)=2$.

\begin{lemma}\label{lem:gddn_rank_2}
	Let $w\leq_{T}\gamma$ with $\ell_{T}(w)=2$. There exists a unique increasing reduced $T$-word of $w$ with respect to the 
	restriction of $\prec_{\gamma}$ to the reflections in $T_{\gamma}\cap [\varepsilon,w]$.
\end{lemma}
\begin{proof}
	Let $w=t_{1}t_{2}$ for $t_{1},t_{2}\in T_{\gamma}$. If $t_{1}$ and $t_{2}$ commute, then $w=t_{1}t_{2}=t_{2}t_{1}$ are the only possible 
	reduced $T$-words of $w$. Since $\prec_{\gamma}$ is a total order there is nothing to show. Suppose that $t_{1}$ and $t_{2}$ do not 
	commute. With the help of Proposition~\ref{prop:gamma_reflections} we can explicitly 
	determine the possible forms of $w$. Analogously to the proof of Proposition~\ref{prop:gamma_reflections}, we investigate the fixed 
	space of $w^{-1}\gamma$ to determine which of these possibilities can actually occur in $\nc{G(d,d,n)}(\gamma)$. The details of this 
	investigation can be found in Appendix~\ref{app:nc_gddn_details_1}. We state here only the relevant cases.
	
	(i) Let $t_{1}=\colref{a}{b}{0},t_{2}=\colref{b}{c}{0}$, where $1\leq a<b<c<n$. We have 
	$w=\bigl(\!\bigl(\colint{a}{0}\;\colint{b}{0}\;\colint{c}{0}\bigr)\!\bigr)$, and the reduced $T$-words of $w$ are
	\begin{align*}
		w & = \colref{a}{b}{0}\colref{b}{c}{0}\\
		& = \colref{b}{c}{0}\colref{a}{c}{0}\\
		& = \colref{a}{c}{0}\colref{a}{b}{0}.
	\end{align*}
	According to \eqref{eq:order_gddn} only $w=\colref{a}{b}{0}\colref{b}{c}{0}$ is increasing.

	(ii) Let $t_{1}=\colref{a}{b}{0},t_{2}=\colref{b}{c}{d-1}$, where $1\leq a<b<c<n$. We have 
	$w=\bigl(\!\bigl(\colint{a}{0}\;\colint{b}{0}\;\colint{c}{d-1}\bigr)\!\bigr)$, and the reduced $T$-words of $w$ are
	\begin{align*}
		w & = \colref{a}{b}{0}\colref{b}{c}{d-1}\\
		& = \colref{b}{c}{d-1}\colref{a}{c}{d-1}\\
		& = \colref{a}{c}{d-1}\colref{a}{b}{0}.
	\end{align*}
	According to \eqref{eq:order_gddn} only $w=\colref{a}{b}{0}\colref{b}{c}{d-1}$ is increasing.
	
	(iii) Let $t_{1}=\colref{a}{b}{0},t_{2}=\colref{b}{n}{s}$, where $1\leq a<b<n$ and $0\leq s<d$. We have 
	$w=\bigl(\!\bigl(\colint{a}{0}\;\colint{b}{0}\;\colint{n}{s}\bigr)\!\bigr)$, and the reduced $T$-words of $w$ are 
	\begin{align*}
		w & = \colref{a}{b}{0}\colref{b}{n}{s}\\
		& = \colref{b}{n}{s}\colref{a}{n}{s}\\
		& = \colref{a}{n}{s}\colref{a}{b}{0}.
	\end{align*}
	According to \eqref{eq:order_gddn} only $w=\colref{a}{b}{0}\colref{b}{n}{s}$ is increasing.

	(iv) Let $t_{1}=\colref{b}{c}{0},t_{2}=\colref{a}{c}{d-1}$, where $1\leq a<b<c<n$. We have 
	$w=\bigl(\!\bigl(\colint{a}{0}\;\colint{b}{d-1}\;\colint{c}{d-1}\bigr)\!\bigr)$, and the reduced $T$-words $w$ are 
	\begin{align*}
		w & = \colref{a}{b}{d-1}\colref{b}{c}{0}\\
		& = \colref{b}{c}{0}\colref{a}{c}{d-1}\\
		& = \colref{a}{c}{d-1}\colref{a}{b}{d-1}.
	\end{align*}
	According to \eqref{eq:order_gddn} only $w=\colref{b}{c}{0}\colref{a}{c}{d-1}$ is increasing.
	
	(v) Let $t_{1}=\colref{a}{b}{d-1},t_{2}=\colref{a}{n}{s}$, where $1\leq a<b<n$ and $0\leq s<d$. We have 
	$w=\bigl(\!\bigl(\colint{a}{0}\;\colint{n}{s}\;\colint{b}{d-1}\bigr)\!\bigr)$, and the reduced $T$-words of $w$ are 
	\begin{align*}
		w & = \colref{a}{n}{s}\colref{b}{n}{s+1}\\
		& = \colref{b}{n}{s+1}\colref{a}{b}{d-1}\\
		& = \colref{a}{b}{d-1}\colref{a}{n}{s}.
	\end{align*}
	According to \eqref{eq:order_gddn} only $w=\colref{a}{n}{s}\colref{b}{n}{s+1}$ is increasing.

	(vi) Let $t_{1}=\colref{a}{n}{s},t_{2}=\colref{a}{n}{t}$, where $1\leq a<n$ and $0\leq s,t<d$ with $t\neq s$. We have 
	$w=\bigl[\colint{a}{0}\bigr]_{t-s}\bigl[\colint{n}{0}\bigr]_{s-t}$, and the reduced $T$-words of $w$ are 
	\begin{align*}
		w & = \colref{a}{n}{s}\colref{a}{n}{s+1}\\
		& = \colref{a}{n}{s+1}\colref{a}{n}{s+2}\\ 
		& = \colref{a}{n}{s+2}\colref{a}{n}{s+3}\\
		& = \quad\cdots\\
		& = \colref{a}{n}{s-1}\colref{a}{n}{s}.
	\end{align*}
	According to \eqref{eq:order_gddn} only $w=\colref{a}{n}{0}\colref{a}{n}{1}$ is increasing.
\end{proof}

As a next step, we show that the restriction of $\lambda_{\gamma}$ to parabolic subgroups isomorphic to $G(1,1,n')$ for $n'\leq n$, yields an EL-labeling
of the corresponding interval in $\nc{G(d,d,n)}$, with respect to the restriction of the order in \eqref{eq:order_gddn}.

\begin{proposition}\label{prop:gddn_symmetric_irreducible}
	Let $w\leq_{T}\gamma$ be such that the parabolic subgroup of $G(d,d,n)$, in which $w$ is a Coxeter element, is isomorphic to $G(1,1,n')$ for some 
	$n'\leq n$. Then $w$ is of one of the following three forms:
	\begin{enumerate}[(i)]
		\item $w=\bigl(\!\bigl(\colint{(a\!+\!1)}{0}\;\colint{(a\!+\!2)}{0}\;\ldots\;\colint{b}{0}\bigr)\!\bigr)$, where $1\leq a<b<n$,
		\item $w=\bigl(\!\bigl(\colint{1}{0}\;\colint{2}{0}\;\ldots\;\colint{a}{0}\;\colint{(b\!+\!1)}{d-1}\;\colint{(b\!+\!2)}{d-1}\;\ldots\;
			\colint{(n\!-\!1)}{d-1}\bigr)\!\bigr)$, where $1\leq a<b<n$, \quad or
		\item $w=\bigl(\!\bigl(\colint{1}{0}\;\colint{2}{0}\;\ldots\;\colint{a}{0}\;\colint{n}{s-1}\colint{(a\!+\!1)}{d-1}\;\colint{(a\!+\!2)}{d-1}\;
			\ldots\;\colint{(n\!-\!1)}{d-1}\bigr)\!\bigr)$, where $1\leq a<n$.
	\end{enumerate}
	Moreover, in each of these cases there exists a unique increasing reduced $T$-word of $w$ with respect to the restriction of $\prec_{\gamma}$ 
	to the reflections in $T_{\gamma}\cap[\varepsilon,w]$.
\end{proposition}
\begin{proof}
	The observation that $w$ can only be of the forms (i)--(iii) is a straightforward computation using 
	Proposition~\ref{prop:gamma_reflections}. The proof of the second part of this proposition is rather technical, and hence omitted here. 
	The details can be found in Appendix~\ref{app:nc_gddn_details_2}. We only present the unique increasing reduced $T$-words of $w$ for the different
	cases:
	
	(i) Let $w=\bigl(\!\bigl(\colint{(a\!+\!1)}{0}\;\colint{(a\!+\!2)}{0}\;\ldots\;\colint{b}{0}\bigr)\!\bigr)$, where $1\leq a<b<n$. The unique increasing
	reduced $T$-word of $w$ is
	\begin{displaymath}
		w = \colref{(a\!+\!1)}{(a\!+\!2)}{0}\colref{(a\!+\!2)}{(a\!+\!3)}{0}\cdots\colref{(b\!-\!1)}{b}{0}.
	\end{displaymath}

	(ii) Let $w=\bigl(\!\bigl(\colint{1}{0}\;\colint{2}{0}\;\ldots\;\colint{a}{0}\;\colint{(b\!+\!1)}{d-1}\;\colint{(b\!+\!2)}{d-1}\;\ldots\;
	\colint{(n\!-\!1)}{d-1}\bigr)\!\bigr)$, where $1\leq a<b<n$. The unique increasing reduced $T$-word of $w$ is
	\begin{multline*}
		w = \colref{1}{2}{0}\colref{2}{3}{0}\cdots\colref{(a\!-\!1)}{a}{0}\\
			\colref{(b\!+\!1)}{(b\!+\!2)}{0}\cdots\colref{(n\!-\!2)}{(n\!-\!1)}{0}\colref{a}{(n\!-\!1)}{d-1}.
	\end{multline*}

	(iii) Let $w=\bigl(\!\bigl(\colint{1}{0}\;\colint{2}{0}\;\ldots\;\colint{a}{0}\;\colint{n}{s-1}\colint{(a\!+\!1)}{d-1}\;\colint{(a\!+\!2)}{d-1}\;
	\ldots\;\colint{(n\!-\!1)}{d-1}\bigr)\!\bigr)$, where $1\leq a<n$. The unique increasing reduced $T$-word of $w$ is
	\begin{align*}
		w & = \colref{1}{2}{0}\colref{2}{3}{0}\cdots\colref{(a\!-\!1)}{a}{0}\colref{(a\!+\!1)}{(a\!+\!2)}{0}\\
		& \kern1cm\colref{(a\!+\!2)}{(a\!+\!3)}{0}\cdots\colref{(n\!-\!2)}{(n\!-\!1)}{0}\colref{a}{n}{s-1}\\
		& \kern1cm\colref{(n\!-\!1)}{n}{s}.
	\end{align*}
\end{proof}

The following corollary is immediate.

\begin{corollary}\label{cor:gddn_symmetric_reducible}
	Let $w\leq_{T}\gamma$ such that the parabolic subgroup $W$ of $G(d,d,n)$, in which $w$ is a Coxeter element, is reducible, and hence 
	$W=W_{1}\times W_{2}\times\cdots\times W_{l}$ for some $l$. If for each $i\in\{1,2,\ldots,l\}$, the group $W_{i}$ is isomorphic to $G(1,1,n_{i})$ 
	for $n_{i}\leq n$, then there exists a unique increasing reduced $T$-word of $w$ with respect to the restriction of $\prec_{\gamma}$ to the 
	reflections in $T_{\gamma}\cap[\varepsilon,w]$.
\end{corollary}
\begin{proof}
	This works analogously to the proof of Proposition~\ref{prop:gddn_symmetric_irreducible}. See Appendix~\ref{app:nc_gddn_details_3} for the details.
\end{proof}

The next result is the analogue of Proposition~\ref{prop:gddn_symmetric_irreducible} for the intervals which are isomorphic to some $G(d,d,n')$ with $n'<n$.

\begin{proposition}\label{prop:gddn_gend_irreducible}
	Let $w\leq_{T}\gamma$ such that the parabolic subgroup of $G(d,d,n)$, in which $w$ is a Coxeter element, is isomorphic to $G(d,d,n')$ for some 
	$n'<n$. There exists a unique increasing reduced $T$-word of $w$ with respect to the restriction of $\prec_{\gamma}$ to the reflections 
	in $T_{\gamma}\cap[\varepsilon,w]$.
\end{proposition}
\begin{proof}
	Again we proceed by induction on $\ell_{T}(w)$, and the case $\ell_{T}(w)=2$ is covered by Lemma~\ref{lem:gddn_rank_2}. In view of 
	Lemma~\ref{lem:interval_isomorphism}, we can assume that $w=\gamma$, and that the claim is true for all $w'<_{T}w$ that satisfy the condition. 
	We notice immediately that the simple decomposition of $\gamma$, namely
	\begin{displaymath}
		\gamma = s_{1}s_{2}\cdots s_{n} = \colref{1}{2}{0}\colref{2}{3}{0}\cdots\colref{(n\!-\!1)}{n}{0}\colref{(n\!-\!1)}{n}{1}
	\end{displaymath}
	is increasing with respect to \eqref{eq:order_gddn}. Let $\gamma=t_{1}t_{2}\cdots t_{n}$ be an increasing
	reduced $T$-word of $\gamma$ that is different from $s_{1}s_{2}\cdots s_{n}$, and let $k$ be the maximal index where $t_{k}\neq s_{k}$. 
	If $k<n$, then $\gamma s_{n}s_{n-1}\cdots s_{k+1}=\bigl(\!\bigl(\colint{1}{0}\;\colint{2}{0}\;\ldots\;\colint{(k\!+\!1)}{0}\bigr)\!\bigr)$. It 
	follows from Proposition~\ref{prop:gddn_symmetric_irreducible} that the only increasing reduced $T$-word of 
	$\gamma s_{n}s_{n-1}\cdots s_{k+1}$
	is $s_{1}s_{2}\cdots s_{k}$, which is a contradiction. Hence let $k=n$. In view of Proposition~\ref{prop:gamma_reflections}, there are 
	essentially three possible choices of $t_{n}$, and we write $\gamma'=\gamma t_{n}$. Moreover, let $W$ denote the parabolic subgroup of $G(d,d,n)$
	in which $\gamma'$ is a Coxeter element. 
	
	(i) Let $t_{n}=\colref{a}{b}{0}$, where $1\leq a<b<n$. Lemma~\ref{lem:gamma_coatoms} implies that we can write $\gamma'=w_{1}w_{2}$ with 
	\begin{align*}
		w_{1} & = \bigl[\colint{1}{0}\;\colint{2}{0}\;\ldots\;\colint{a}{0}\;\colint{(b\!+\!1)}{0}\;\colint{(b\!+\!2)}{0}\;\ldots\;
			\colint{(n\!-\!1)}{0}\bigr]\bigl[\colint{n}{0}\bigr]_{d-1},\quad\text{and}\\
		w_{2} & = \bigl(\!\bigl(\colint{(a\!+\!1)}{0}\;\colint{(a\!+\!2)}{0}\;\ldots\;\colint{b}{0}\bigr)\!\bigr).
	\end{align*}
	This implies that $w_{1}$ is a Coxeter element in a parabolic subgroup $W_{1}$ of $G(d,d,n)$ isomorphic to $G(d,d,n+a-b)$, and $w_{2}$ is a 
	Coxeter element in a parabolic subgroup $W_{2}$ of $G(d,d,n)$ isomorphic to $G(1,1,b-a-1)$, and we can write $W=W_{1}\times W_{2}$. By induction 
	hypothesis and by Proposition~\ref{prop:gddn_symmetric_irreducible} there exist unique increasing reduced $T$-words of $w_{1}$ and $w_{2}$, 
	namely
	\begin{align*}
		w_{1} & = \colref{1}{2}{0}\colref{2}{3}{0}\cdots\colref{(a\!-\!1)}{a}{0}\colref{a}{(b\!+\!1)}{0}\\
		& \kern1cm\colref{(b\!+\!1)}{(b\!+\!2)}{0}\cdots\colref{(n\!-\!1)}{n}{0}\colref{(n\!-\!1)}{n}{1},\quad\text{and}\\
		w_{2} & = \colref{(a\!+\!1)}{(a\!+\!2)}{0}\colref{(a\!+\!2)}{(a\!+\!3)}{0}\cdots\colref{(b\!-\!1)}{b}{0}.
	\end{align*}
	It is immediate to see that 
	\begin{multline*}
		\gamma' = \colref{1}{2}{0}\cdots\colref{(a\!-\!1)}{a}{0}\colref{a}{(b\!+\!1)}{0}\colref{(a\!+\!1)}{(a\!+\!2)}{0}\cdots\\
			\colref{(b\!-\!1)}{b}{0}\colref{(b\!+\!1)}{(b\!+\!2)}{0}\cdots\colref{(n\!-\!1)}{n}{0}\colref{(n\!-\!1)}{n}{1}
	\end{multline*}
	is the unique increasing reduced $T$-word of $\gamma'$ and hence has to correspond to $t_{1}t_{2}\cdots t_{n-1}$. However, we have for instance
	$\colref{(n\!-\!1)}{n}{1}\succ_{\gamma}\colref{a}{b}{0}=t_{n}$, which contradicts the assumption that $t_{1}t_{2}\cdots t_{n}$ is increasing.
	
	(ii) Let $t_{n}=\colref{a}{b}{d-1}$, where $1\leq a<b<n$. Lemma~\ref{lem:gamma_coatoms} implies that we can write $\gamma'=w_{1}w_{2}$ with 
	\begin{align*}
		w_{1} & = \bigl(\!\bigl(\colint{1}{0}\;\colint{2}{0}\;\ldots\;\colint{a}{0}\;\colint{(b\!+\!1)}{d-1}\;\colint{(b\!+\!2)}{d-1}\;\ldots\;
			\colint{(n\!-\!1)}{d-1}\bigr)\!\bigr),\quad\text{and}\\
		w_{2} & = \bigl[\colint{(a\!+\!1)}{0}\;\colint{(a\!+\!2)}{0}\;\ldots\;\colint{b}{0}\bigr]\bigl[\colint{n}{0}\bigr]_{d-1}.
	\end{align*}
	This implies that $w_{1}$ is a Coxeter element in a parabolic subgroup $W_{1}$ of $G(d,d,n)$ isomorphic to $G(1,1,n+a-b-2)$, and $w_{2}$ is a 
	Coxeter element in a parabolic subgroup $W_{2}$ of $G(d,d,n)$ isomorphic to $G(d,d,b-a+1)$, and we can write $W=W_{1}\times W_{2}$. By induction 
	hypothesis and by Proposition~\ref{prop:gddn_symmetric_irreducible} there exist unique increasing reduced $T$-words of $w_{1}$ and $w_{2}$, 
	namely
	\begin{align*}
		w_{1} & = \colref{1}{2}{0}\colref{2}{3}{0}\cdots\colref{(a\!-\!1)}{a}{0}\\
		& \kern1cm\colref{(b\!+\!1)}{(b\!+\!2)}{0}\cdots\colref{(n\!-\!2)}{(n\!-\!1)}{0}\colref{a}{(n\!-\!1)}{d-1},\quad\text{and}\\
		w_{2} & = \colref{(a\!+\!1)}{(a\!+\!2)}{0}\colref{(a\!+\!2)}{(a\!+\!3)}{0}\cdots\colref{(b\!-\!1)}{b}{0}\\
		& \kern1cm\colref{b}{n}{0}\colref{b}{n}{1}.
	\end{align*}
	It is immediate to see that 
	\begin{align*}
		\gamma' & = \colref{1}{2}{0}\cdots\colref{(a\!-\!1)}{a}{0}\colref{(a\!+\!1)}{(a\!+\!2)}{0}\cdots\colref{(b\!-\!1)}{b}{0}\\
		& \kern1cm\colref{(b\!+\!1)}{(b\!+\!2)}{0}\cdots\colref{(n\!-\!2)}{(n\!-\!1)}{0}\colref{a}{(n\!-\!1)}{d-1}\\
		& \kern1cm\colref{b}{n}{0}\colref{b}{n}{1}
	\end{align*}
	is the unique increasing reduced $T$-word of $\gamma'$ and hence has to correspond to $t_{1}t_{2}\cdots t_{n-1}$. However, we have for instance
	$\colref{b}{n}{1}\succ_{\gamma}\colref{a}{b}{d-1}=t_{n}$, which contradicts the assumption that $t_{1}t_{2}\cdots t_{n}$ is increasing.
	
	(iii) Let $t=\colref{a}{n}{s}$, where $1\leq a<n-1$ and $0\leq s<d$. Lemma~\ref{lem:gamma_coatoms} implies that we can write 
	\begin{displaymath}
		\gamma' = \bigl(\!\bigl(\colint{1}{0}\;\colint{2}{0}\;\ldots\;\colint{a}{0}\;\colint{n}{s-1}\colint{(a\!+\!1)}{d-1}\;
			\colint{(a\!+\!2)}{d-1}\;\ldots\;\colint{(n\!-\!1)}{d-1}\bigr)\!\bigr).
	\end{displaymath}
	In view of Proposition~\ref{prop:gddn_symmetric_irreducible} there exists a unique increasing reduced $T$-word of $\gamma$, namely
	\begin{multline*}
		\gamma' = \colref{1}{2}{0}\cdots\colref{(a\!-\!1)}{a}{0}\colref{(a\!+\!1)}{(a\!+\!2)}{0}\cdots\\
			\colref{(n\!-\!2)}{(n\!-\!1)}{0}\colref{a}{n}{s-1}\colref{(n\!-\!1)}{n}{s},
	\end{multline*}
	and this word has to correspond to $t_{1}t_{2}\cdots t_{n-1}$. However, we have for instance
	$\colref{(n\!-\!1)}{n}{s}\succ_{\gamma}\colref{a}{n}{s}=t_{n}$, which contradicts the assumption that $t_{1}t_{2}\cdots t_{n}$ is increasing.
	
	(iv) Let $t=\colref{(n\!-\!1)}{n}{s}$, where $0\leq s<d$. It follows that $s\neq 1$, because otherwise $t_{n}=s_{n}$. 
	Lemma~\ref{lem:gamma_coatoms} implies
	that we can write 
	\begin{displaymath}
		\gamma' = \bigl(\!\bigl(\colint{1}{0}\;\colint{2}{0}\;\ldots\;\colint{(n\!-\!1)}{0}\;\colint{n}{s-1}\bigr)\!\bigr).
	\end{displaymath}
	In view of Proposition~\ref{prop:gddn_symmetric_irreducible} there exists a unique increasing reduced $T$-word of $\gamma$, namely
	\begin{displaymath}
		\gamma' = \colref{1}{2}{0}\cdots\colref{(n\!-\!2)}{(n\!-\!1)}{0}\colref{(n\!-\!1)}{n}{s-1},
	\end{displaymath}
	and this word has to correspond to $t_{1}t_{2}\cdots t_{n-1}$. However, since $s\neq 1$, we have for instance 
	$\colref{(n\!-\!1)}{n}{s-1}\succ_{\gamma}\colref{(n\!-\!1)}{n}{s}=t_{n}$, which contradicts the assumption that $t_{1}t_{2}\cdots t_{n}$ is increasing.
	
	Hence $\gamma=s_{1}s_{2}\cdots s_{n}$ is the unique increasing reduced $T$-word of $\gamma$.
\end{proof}

The following corollary is immediate.

\begin{corollary}\label{cor:gddn_gend_reducible}
	Let $w\leq_{T}\gamma$ such that the parabolic subgroup $W$ of $G(d,d,n)$, in which $w$ is a Coxeter element, is reducible. There exists a 
	unique increasing reduced $T$-word of $w$ with respect to the restriction of $\prec_{\gamma}$ to the reflections in 
	$T_{\gamma}\cap[\varepsilon,w]$.
\end{corollary}
\begin{proof}
	Since $W$ is reducible, we can write $W=W_{1}\times W_{2}\times\cdots\times W_{l}$ for some $l$. It follows for instance from 
	\cite{broue98complex}*{Fact~1.7} and \cite{broue98complex}*{Table~2} that at most one $W_{i}$ is isomorphic to $G(d,d,n')$ for some $n'<n$, and the 
	other $W_{j}$ are isomorphic to $G(1,1,n_{j})$ for $n_{j}\leq n$. The proof works analogously to the proofs of 
	Corollary~\ref{cor:gddn_symmetric_reducible} and Proposition~\ref{prop:gddn_gend_irreducible}.
\end{proof}

Now we have collected all the ingredients for the proof of Theorem~\ref{thm:shellability_gddn}.

\begin{proof}[Proof of Theorem~\ref{thm:shellability_gddn}]
	We need to show that under the given assumptions in every interval $[u,v]$ of $\nc{G(d,d,n)}(\gamma)$ there exists a unique increasing maximal chain, 
	and this maximal chain is lexicographically first. In view of Lemma~\ref{lem:interval_isomorphism}, it suffices to consider intervals of the form
	$[\varepsilon,w]$, and Proposition~\ref{prop:increasing_chain} implies that the lexicographically first maximal chain in $[\varepsilon,w]$ is 
	increasing. Now, 
	Propositions~\ref{prop:gddn_symmetric_irreducible} and \ref{prop:gddn_gend_irreducible} as well as 
	Corollaries~\ref{cor:gddn_symmetric_reducible} and \ref{cor:gddn_gend_reducible} imply together with Lemma~\ref{lem:reduced_word_maximal_chain}
	that there is exactly one increasing maximal chain in $[\varepsilon,w]$, and we are done.
\end{proof}

\begin{example}\label{ex:g553_1}
	Let us consider the group $G(5,5,3)$. The Coxeter element $\gamma$ according to \eqref{eq:coxeter_gddn} is 
	$\gamma=\bigl[\colint{1}{0}\;\colint{2}{0}\bigr]\bigl[\colint{3}{0}\bigr]_{d-1}$, and the reflections in $T_{\gamma}$ 
	are 
	\begin{displaymath}\begin{aligned}
		& \colref{1}{2}{0}, && \colref{1}{2}{4}, && \colref{1}{3}{0}, && \colref{1}{3}{1}, \\
		& \colref{1}{3}{2}, && \colref{1}{3}{3}, && \colref{1}{3}{4}, && \colref{2}{3}{0}, \\
		& \colref{2}{3}{1}, && \colref{2}{3}{2}, && \colref{2}{3}{3}, && \colref{2}{3}{4}.
	\end{aligned}\end{displaymath}
	The total order of $T_{\gamma}$ according to \eqref{eq:order_gddn} is 
	\begin{align}\label{eq:order_g553}
		\colref{1}{2}{0} & \prec\colref{1}{3}{0}\prec\colref{1}{3}{4}\prec\colref{1}{3}{3}\prec\colref{1}{3}{2}\\
		&	\prec\colref{1}{3}{1}\prec\colref{1}{2}{4}\prec\colref{2}{3}{0}\prec\colref{2}{3}{4}\nonumber\\ 
		& \prec\colref{2}{3}{3}\prec\colref{2}{3}{2}\prec\colref{2}{3}{1}.\nonumber
	\end{align}
	Figure~\ref{fig:nc_g553} shows the lattice $\nc{G(5,5,3)}$. The given integer labeling is derived from 
	$\lambda_{\gamma}$ by mapping every reflection to its position in the total order given in
	\eqref{eq:order_g553}. The nodes of this lattice are labeled by products of integers, which 
	correspond to products of the corresponding reflections under the mapping explained before. For instance,
	the label $1\!\cdot\!10$ represents the product $\colref{1}{2}{0}\colref{2}{3}{3}$. We can quickly check that this 
	is indeed an EL-labeling, where the unique increasing chain in the interval $[\varepsilon,\gamma]$ is indicated by 
	thick edges. 
	
	\begin{figure}
		\centering
		\begin{tikzpicture}\small
			\def\x{1.1};
			\def\y{4};
			\draw(6.5*\x,.5*\y) node(n1){$\varepsilon$};
			\draw(1*\x,1*\y) node(n2){$8$};
			\draw(2*\x,1*\y) node(n3){$2$};
			\draw(3*\x,1*\y) node(n4){$7$};
			\draw(4*\x,1*\y) node(n5){$3$};
			\draw(5*\x,1*\y) node(n6){$12$};
			\draw(6*\x,1*\y) node(n7){$9$};
			\draw(7*\x,1*\y) node(n8){$6$};
			\draw(8*\x,1*\y) node(n9){$4$};
			\draw(9*\x,1*\y) node(n10){$11$};
			\draw(10*\x,1*\y) node(n11){$1$};
			\draw(11*\x,1*\y) node(n12){$10$};
			\draw(12*\x,1*\y) node(n13){$5$};
			\draw(1*\x,2*\y) node(n14){$1\!\cdot\!8$};
			\draw(2*\x,2*\y) node(n15){$3\!\cdot\!8$};
			\draw(3*\x,2*\y) node(n16){$2\!\cdot\!6$};
			\draw(4*\x,2*\y) node(n17){$2\!\cdot\!12$};
			\draw(5*\x,2*\y) node(n18){$1\!\cdot\!9$};
			\draw(6*\x,2*\y) node(n19){$1\!\cdot\!12$};
			\draw(7*\x,2*\y) node(n20){$4\!\cdot\!9$};
			\draw(8*\x,2*\y) node(n21){$6\!\cdot\!11$};
			\draw(9*\x,2*\y) node(n22){$1\!\cdot\!10$};
			\draw(10*\x,2*\y) node(n23){$8\!\cdot\!12$};
			\draw(11*\x,2*\y) node(n24){$1\!\cdot\!11$};
			\draw(12*\x,2*\y) node(n25){$5\!\cdot\!10$};
			\draw(6.5*\x,2.5*\y) node(n26){$1\!\cdot\!8\!\cdot\!12$};	
			\draw(n1) -- (n2) node[fill=white] at(3.2*\x,.8*\y){\tiny $8$};
			\draw(n1) -- (n3) node[fill=white] at(3.8*\x,.8*\y){\tiny $2$};
			\draw(n1) -- (n4) node[fill=white] at(4.4*\x,.8*\y){\tiny $7$};
			\draw(n1) -- (n5) node[fill=white] at(5*\x,.8*\y){\tiny $3$};
			\draw(n1) -- (n6) node[fill=white] at(5.6*\x,.8*\y){\tiny $12$};
			\draw(n1) -- (n7) node[fill=white] at(6.2*\x,.8*\y){\tiny $9$};
			\draw(n1) -- (n8) node[fill=white] at(6.8*\x,.8*\y){\tiny $6$};
			\draw(n1) -- (n9) node[fill=white] at(7.4*\x,.8*\y){\tiny $4$};
			\draw(n1) -- (n10) node[fill=white] at(8*\x,.8*\y){\tiny $11$};
			\draw[very thick](n1) -- (n11) node[fill=white] at(8.6*\x,.8*\y){\tiny $1$};
			\draw(n1) -- (n12) node[fill=white] at(9.2*\x,.8*\y){\tiny $10$};
			\draw(n1) -- (n13) node[fill=white] at(9.8*\x,.8*\y){\tiny $5$};
			\draw(n2) -- (n14) node[fill=white] at(1*\x,1.5*\y){\tiny $2$};
			\draw(n2) -- (n15) node[fill=white] at(1.25*\x,1.25*\y){\tiny $7$};
			\draw(n2) -- (n23) node[fill=white] at(1.55*\x,1.06*\y){\tiny $12$};
			\draw(n3) -- (n14) node[fill=white] at(1.25*\x,1.75*\y){\tiny $1$};
			\draw(n3) -- (n16) node[fill=white] at(2.4*\x,1.4*\y){\tiny $6$};
			\draw(n3) -- (n17) node[fill=white] at(3.2*\x,1.6*\y){\tiny $12$};
			\draw(n4) -- (n15) node[fill=white] at(2.4*\x,1.6*\y){\tiny $3$};
			\draw(n4) -- (n17) node[fill=white] at(3.15*\x,1.15*\y){\tiny $2$};
			\draw(n4) -- (n20) node[fill=white] at(4*\x,1.27*\y){\tiny $4$};
			\draw(n4) -- (n21) node[fill=white] at(5*\x,1.38*\y){\tiny $6$};
			\draw(n4) -- (n25) node[fill=white] at(11.45*\x,1.94*\y){\tiny $5$};
			\draw(n5) -- (n15) node[fill=white] at(3.2*\x,1.4*\y){\tiny $8$};
			\draw(n5) -- (n16) node[fill=white] at(3.15*\x,1.85*\y){\tiny $2$};
			\draw(n5) -- (n18) node[fill=white] at(4.75*\x,1.75*\y){\tiny $1$};
			\draw(n6) -- (n17) node[fill=white] at(4.75*\x,1.25*\y){\tiny $7$};
			\draw(n6) -- (n19) node[fill=white] at(5.15*\x,1.15*\y){\tiny $6$};
			\draw(n6) -- (n23) node[fill=white] at(8*\x,1.62*\y){\tiny $11$};
			\draw(n7) -- (n18) node[fill=white] at(5.15*\x,1.85*\y){\tiny $3$};
			\draw(n7) -- (n20) node[fill=white] at(6.1*\x,1.1*\y){\tiny $7$};
			\draw(n7) -- (n23) node[fill=white] at(9*\x,1.73*\y){\tiny $8$};
			\draw(n8) -- (n16) node[fill=white] at(4*\x,1.73*\y){\tiny $5$};
			\draw(n8) -- (n19) node[fill=white] at(6.1*\x,1.9*\y){\tiny $1$};
			\draw(n8) -- (n21) node[fill=white] at(7.85*\x,1.85*\y){\tiny $11$};
			\draw(n9) -- (n16) node[fill=white] at(5*\x,1.62*\y){\tiny $3$};
			\draw(n9) -- (n20) node[fill=white] at(7.85*\x,1.15*\y){\tiny $9$};
			\draw(n9) -- (n22) node[fill=white] at(8.25*\x,1.25*\y){\tiny $1$};
			\draw(n10) -- (n21) node[fill=white] at(8.25*\x,1.75*\y){\tiny $7$};
			\draw(n10) -- (n23) node[fill=white] at(9.85*\x,1.85*\y){\tiny $10$};
			\draw(n10) -- (n24) node[fill=white] at(9.8*\x,1.4*\y){\tiny $5$};
			\draw[very thick](n11) -- (n14) node[fill=white] at(1.55*\x,1.94*\y){\tiny $8$};
			\draw(n11) -- (n18) node[fill=white] at(8*\x,1.38*\y){\tiny $9$};
			\draw(n11) -- (n19) node[fill=white] at(9*\x,1.27*\y){\tiny $12$};
			\draw(n11) -- (n22) node[fill=white] at(9.85*\x,1.15*\y){\tiny $10$};
			\draw(n11) -- (n24) node[fill=white] at(10.6*\x,1.6*\y){\tiny $11$};
			\draw(n12) -- (n22) node[fill=white] at(9.8*\x,1.6*\y){\tiny $4$};
			\draw(n12) -- (n23) node[fill=white] at(10.6*\x,1.4*\y){\tiny $9$};
			\draw(n12) -- (n25) node[fill=white] at(11.75*\x,1.75*\y){\tiny $7$};
			\draw(n13) -- (n16) node[fill=white] at(11.45*\x,1.06*\y){\tiny $4$};
			\draw(n13) -- (n24) node[fill=white] at(11.75*\x,1.25*\y){\tiny $1$};
			\draw(n13) -- (n25) node[fill=white] at(12*\x,1.5*\y){\tiny $10$};
			\draw[very thick](n14) -- (n26) node[fill=white] at(3.2*\x,2.2*\y){\tiny $12$};
			\draw(n15) -- (n26) node[fill=white] at(3.8*\x,2.2*\y){\tiny $2$};
			\draw(n16) -- (n26) node[fill=white] at(4.4*\x,2.2*\y){\tiny $1$};
			\draw(n17) -- (n26) node[fill=white] at(5*\x,2.2*\y){\tiny $6$};
			\draw(n18) -- (n26) node[fill=white] at(5.6*\x,2.2*\y){\tiny $8$};
			\draw(n19) -- (n26) node[fill=white] at(6.2*\x,2.2*\y){\tiny $11$};
			\draw(n20) -- (n26) node[fill=white] at(6.8*\x,2.2*\y){\tiny $3$};
			\draw(n21) -- (n26) node[fill=white] at(7.4*\x,2.2*\y){\tiny $5$};
			\draw(n22) -- (n26) node[fill=white] at(8*\x,2.2*\y){\tiny $9$};
			\draw(n23) -- (n26) node[fill=white] at(8.6*\x,2.2*\y){\tiny $7$};
			\draw(n24) -- (n26) node[fill=white] at(9.2*\x,2.2*\y){\tiny $8$};
			\draw(n25) -- (n26) node[fill=white] at(9.8*\x,2.2*\y){\tiny $4$};
		\end{tikzpicture}
		\caption{The lattice of noncrossing partitions of $G(5,5,3)$. The integer labels correspond to the position 
			of the reflections in \eqref{eq:order_g553}.}
		\label{fig:nc_g553}
	\end{figure}
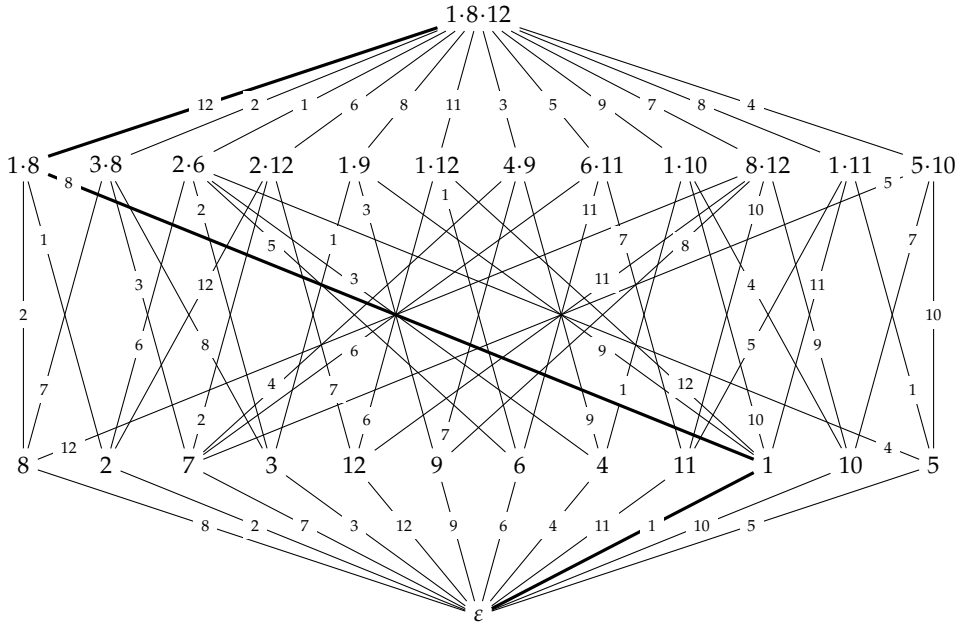
\end{example}

\section{EL-Shellability of $\nc{W}$ for the Exceptional Well-Generated Complex Reflection Groups}
  \label{sec:shellability_exceptional}
As remarked in the beginning of Section~\ref{sec:gddn_labeling}, an analogue of the edge-labeling of $\nc{G(d,d,n)}$ used in 
the proof of Theorem~\ref{thm:shellability_gddn} was already used by Athanasiadis, Brady and Watt in 
\cite{athanasiadis07shellability}, and it has a natural connection to the definition of the absolute order. We can 
define this labeling more generally for every well-generated complex reflection group $W$, and every Coxeter element 
$\gamma\in W$ by
\begin{equation}\label{eq:ncw_labeling}
	\lambda_{\gamma}:\mathcal{E}\bigl(\nc{W}(\gamma)\bigr)\to T_{\gamma},\quad (u,v)\mapsto u^{-1}v,
\end{equation} 
where $T_{\gamma}$ denotes the set of reflections of $W$ which are contained in $\nc{W}(\gamma)=[\varepsilon,\gamma]$. In 
this section, we provide explicit orders of $T_{\gamma}$ for $\gamma$ a Coxeter element of an exceptional
well-generated complex reflection group $W$ such that $\lambda_{\gamma}$ is an EL-labeling of the corresponding 
noncrossing partition lattice. 

It turns out that the noncrossing partition lattice of most of these groups is isomorphic
to the noncrossing partition lattice of some real reflection group. Only five groups, namely $G_{24},G_{27},G_{29},G_{33}$
and $G_{34}$, remain unrelated to any known case. For these cases, we have proven EL-shellability by means of a computer
program, called \textsc{Lins}, see \cite{lins}. Given a well-generated complex reflection group $W$, \textsc{Lins} takes 
an arbitrary order of the reflections of $W$ which are contained in $\nc{W}$, and checks which rank two intervals have more
than one increasing chain with respect to this order. It subsequently adapts the order such that only one increasing
chain remains, and checks that the labeling from \eqref{eq:ncw_labeling} is indeed an EL-labeling of $\nc{W}$. However, 
this algorithm is not 
deterministic, meaning that different runs of \textsc{Lins} may produce different orders. It uses Michel's
GAP-distribution \cite{gap}, and Borchmann's FCA-tool \cite{conexpclj} for computing the chains of the lattice.
For more information on Formal Concept Analysis (FCA), we refer to the standard monograph by Ganter and Wille, see 
\cite{ganter99formal}. \textsc{Lins} outputs several files, including some GAP scripts, a file containing the labeled
chains, as well as a file containing the final order of the reflections. The reflections are abstractly named
by $s_{k}$, where $k$ is an integer between $1$ and $\lvert\nc{W}\rvert$. The value $k$ that is assigned to a certain
reflection depends on the position at which GAP lists this reflection in its internal representation of the group 
elements of $W$. This naming of the reflections is deterministic, so that we can identify the actual group element 
behind the names with GAP and the respective GAP script\footnote{There is a file named \texttt{lins} included in the
zip-archive containing \textsc{Lins}. Moreover, this script can be downloaded separately from
\url{http://homepage.univie.ac.at/henri.muehle/files/lins}.}.

The main result of this section is proven in the subsequent paragraphs explicitly.
\begin{theorem}\label{thm:shellability_exceptional}
	The lattice $\nc{W}$ is EL-shellable for every exceptional well-generated complex reflection group $W$.
\end{theorem}

Before we investigate the exceptional well-generated complex reflection groups, we need one more observation.
Recall that the \alert{braid group} associated with a complex reflection group $W$, denoted by $\mathfrak{B}(W)$, is 
the fundamental group of the complement of the hyperplanes of $W$. We have the following result.
\begin{proposition}\label{prop:correspondence_shephard_groups}
	For $d,n\geq 2$, we have $\nc{G(d,1,n)}\cong\nc{B_{n}}$. Moreover, we have 
	$\nc{G_{25}}\cong\nc{A_{3}},\nc{G_{26}}\cong\nc{B_{3}}$, and $\nc{G_{32}}\cong\nc{A_{4}}$.
\end{proposition}
\begin{proof}
	It follows for instance from \cite{broue98complex}*{Table~1} that 
	$\mathfrak{B}\bigl(G(d,1,n)\bigr)\cong\mathfrak{B}(B_{n})$ for $n\geq 1$ and $d\geq 2$, as well as 
	$\mathfrak{B}(G_{25})\cong\mathfrak{B}(A_{3}),\mathfrak{B}(G_{26})\cong\mathfrak{B}(B_{3})$, and 
	$\mathfrak{B}(G_{32})\cong\mathfrak{B}(A_{4})$. Following \cite{bessis14finite}, for an irreducible 
	well-generated complex reflection group, we can view $\nc{W}$ as a poset of the so-called simple elements 
	of $\mathfrak{B}(W)$. Since the braid groups of the groups in question are isomorphic, so are their sets of
	simple elements, and the result follows.
\end{proof}

We remark that the close structural connection between the groups in 
Proposition~\ref{prop:correspondence_shephard_groups} was already observed in \cite{orlik88discriminants}.

\subsubsection*{The Groups $G_{23},G_{28},G_{30},G_{35},G_{36},G_{37}$.} These groups are the six exceptional real 
reflection groups, see \cite{broue98complex}*{p.\;6}. Hence their noncrossing partition lattices are EL-shellable
by Theorem~\ref{thm:shellability_real}.
  
\subsubsection*{The Groups $G_{25},G_{26},G_{32}$.} Proposition~\ref{prop:correspondence_shephard_groups} states that
the noncrossing partition lattices of these groups are isomorphic to noncrossing partition lattices of real reflection
groups, and their EL-shellability follows again from Theorem~\ref{thm:shellability_real}. 

\subsubsection*{The Groups $G_{4},G_{5},G_{6},G_{8},G_{9},G_{10},G_{14},G_{16},G_{17},G_{18},G_{20},G_{21}$.}
These are the exceptional well-generated complex reflection groups of rank two. Hence the corresponding lattices of
noncrossing partitions have rank two as well, and are thus isomorphic to a lattice of noncrossing partitions of a 
dihedral group, and their EL-shellability follows again from Theorem~\ref{thm:shellability_real}. 

\subsubsection*{The Groups $G_{24},G_{27},G_{29},G_{33},G_{34}$.} 

\begin{figure}
	\centering
	\begin{tabular}{c|l}
		Group & Reflection order\\
		\hline
		&\\
		$G_{24}$ & $\begin{aligned}s_{26} & \prec s_{5}\prec s_{3}\prec s_{29}\prec s_{21}
		  \prec s_{28}\prec s_{18}\prec s_{7}\prec s_{2}\prec s_{4}\prec s_{11}\prec s_{8}\\
		& \prec s_{23}\prec s_{25}\end{aligned}$\\
		&\\
		\hline
		&\\
		$G_{27}$ & $\begin{aligned}s_{23} & \prec s_{38}\prec s_{42}\prec s_{15}\prec s_{36}
		  \prec s_{29}\prec s_{33}\prec s_{27}\prec s_{18}\prec s_{13}\prec s_{4}\\ 
		&\prec s_{3}\prec s_{2}\prec s_{8}\prec s_{5}\prec s_{21}\prec s_{17}\prec s_{34}\prec s_{37}
		  \prec s_{30}\end{aligned}$\\
		&\\
		\hline
		&\\
		$G_{29}$ & $\begin{aligned}s_{101} & \prec s_{4}\prec s_{76}\prec s_{109}\prec s_{8}
		  \prec s_{105}\prec s_{64}\prec s_{47}\prec s_{6}\prec s_{33}\prec s_{68}\\
		& \prec s_{13}\prec s_{20}\prec s_{39}\prec s_{93}\prec s_{9}\prec s_{88}\prec s_{2}
		  \prec s_{70}\prec s_{28}\prec s_{110}\\
		& \prec s_{25}\prec s_{53}\prec s_{3}\prec s_{18}\end{aligned}$\\
		&\\
		\hline
		&\\
		$G_{33}$ & $\begin{aligned}s_{5} & \prec s_{13}\prec s_{7}\prec s_{33}\prec s_{56}
		  \prec s_{19}\prec s_{36}\prec s_{58}\prec s_{47}\prec s_{182}\prec s_{16}\\ 
		& \prec s_{17}\prec s_{224}\prec s_{281}\prec s_{297}\prec s_{42}\prec s_{179}\prec s_{217}
		  \prec s_{89}\prec s_{128}\\
		& \prec s_{86}\prec s_{110}\prec s_{2}\prec s_{172}\prec s_{277}\prec s_{169}\prec s_{76}
		  \prec s_{68}\prec s_{3}\prec s_{12}
		  \end{aligned}$\\
		&\\
		\hline
		&\\
		$G_{34}$ & $\begin{aligned} s_{1568} & \prec s_{937}\prec s_{1361}\prec s_{213}\prec s_{13}
		  \prec s_{142}\prec s_{669}\prec s_{888}\prec s_{58}\prec s_{7}\\
		& \prec s_{65}\prec s_{67}\prec s_{480}\prec s_{295}\prec s_{8}\prec s_{37}\prec s_{40}
		  \prec s_{256}\prec s_{714}\\
		& \prec s_{1060}\prec s_{1447}\prec s_{17}\prec s_{3}\prec s_{117}\prec s_{53}\prec s_{1252}
		  \prec s_{639}\prec s_{62}\\
		& \prec s_{6}\prec s_{702}\prec s_{915}\prec s_{1043}\prec s_{43}\prec s_{359}\prec s_{428}
		  \prec s_{23}\prec s_{4}\\
		& \prec s_{75}\prec s_{127}\prec s_{191}\prec s_{368}\prec s_{157}\prec s_{648}\prec s_{1234}
		  \prec s_{181}\prec s_{2}\\
		& \prec s_{683}\prec s_{49}\prec s_{264}\prec s_{235}\prec s_{905}\prec s_{1241}\prec s_{60}
		  \prec s_{1558}\prec s_{1353}\\
		& \prec s_{319}\end{aligned}$\\
	\end{tabular}
	\caption{Explicit reflection orders for the remaining groups that make the edge-labeling from \eqref{eq:ncw_labeling}
	  an EL-labeling of the lattice of noncrossing partitions associated with the given group.}
	\label{fig:explicit_orders}
\end{figure}

As described in the beginning of this section, we 
provide an explicit reflection order for these groups that was computed with \textsc{Lins} \cite{lins}. The abstract
encodings listed in Figure~\ref{fig:explicit_orders} can be resolved with the GAP script provided by \textsc{Lins}.
Note that the given reflection orders are just one possibility to make the edge-labeling in \eqref{eq:ncw_labeling} 
an EL-labeling. 

\medskip

Now, we are finally set to prove Theorems~\ref{thm:shellability_main} and \ref{thm:shellability_all}.
\begin{proof}[Proof of Theorem~\ref{thm:shellability_main}]
	This follows from Theorems~\ref{thm:shellability_gddn} and \ref{thm:shellability_exceptional}.
\end{proof}

\begin{proof}[Proof of Theorem~\ref{thm:shellability_all}]
	If $W$ is irreducible, then the result follows from Theorems~\ref{thm:shellability_main}, 
	\ref{thm:shellability_real} and Proposition~\ref{prop:correspondence_shephard_groups}. If $W$ is reducible, 
	then we can write $W\cong W_{1}\times W_{2}\times\cdots\times W_{k}$ where the groups $W_{i}$ are irreducible
	well-generated complex reflection groups for all $i\in\{1,2,\ldots,k\}$. Then by construction, 
	$\nc{W}\cong\nc{W_{1}}\times\nc{W_{2}}\times\cdots\times\nc{W_{k}}$, and the result follows using the first 
	part of this proof and Theorem~\ref{thm:direct_product_shellable}. 
\end{proof}

\subsection{A Uniform Approach?}
  \label{sec:uniform}
In \cite{athanasiadis07shellability}, Athanasiadis, Brady and Watt defined the concept of a $\gamma$-compatible 
reflection order, and used this concept to uniformly prove the EL-shellability of the noncrossing partition lattices 
associated with real reflection groups. However, their definition, see \cite{athanasiadis07shellability}*{Definition~3.1},
used properties of root systems, and these properties cannot be generalized to well-generated complex reflection groups. We propose
the following generalization of their definition.

\begin{definition}\label{def:compatible_order}
	Let $W$ be a well-generated complex reflection group, let $\gamma\in W$ be a Coxeter element, and let $T_{\gamma}$ denote 
	the set of reflections of $W$ which lie below $\gamma$. A total order of $T_{\gamma}$ is called \alert{$\gamma$-compatible}
	if for every rank two interval of $\nc{W}(\gamma)$ there exists a unique increasing chain with respect to the edge-labeling 
	$\lambda_{\gamma}$ from \eqref{eq:ncw_labeling}.
\end{definition}

Thus the reflection orders given in \cite{athanasiadis07shellability}*{Examples~3.2--3.4} are particular instances of our
notion of a $\gamma$-compatible reflection order, and so is the reflection order given in 
\cite{athanasiadis07shellability}*{Theorem~4.1} which implies uniformly that $\gamma$-compatible reflection orders exist for
real reflection groups. Then in view of Proposition~\ref{prop:correspondence_shephard_groups}, the existence of 
$\gamma$-compatible reflection orders follows as well for the groups $G(d,1,n)$, where $d,n\geq 3$. Furthermore, Lemma~\ref{lem:gddn_rank_2}
implies that the order given in \eqref{eq:order_gddn} is compatible with the Coxeter element from \eqref{eq:coxeter_gddn},
which implies the existence of $\gamma$-compatible reflection orders for the groups $G(d,d,n)$, where $d,n\geq 3$. For the 
remaining (exceptional) well-generated complex reflection groups, the existence of $\gamma$-compatible reflection orders can 
be verified with \textsc{Lins}. Now, computer experiments suggest the following conjecture.

\begin{conjecture}\label{conj:compatible_shellable}
	Let $W$ be a well-generated complex reflection group, and let $\gamma\in W$ be a Coxeter element. Let $T_{\gamma}$ denote 
	the reflections in $W$ which are below $\gamma$ in absolute order. If $\prec$ is a $\gamma$-compatible reflection 
	order of $T_{\gamma}$, then $\lambda_{\gamma}$ as defined in \eqref{eq:ncw_labeling} is an EL-labeling of $\nc{W}(\gamma)$.
\end{conjecture}

In particular, if this conjecture was true, it would imply that it is sufficient to check the rank two intervals in order to
derive an EL-labeling of $\nc{W}$. This would yield an immense decrease in the running time of the computation of such a 
labeling.

\begin{remark}
	Substantial progress to an affirmative solution of Conjecture~\ref{conj:compatible_shellable} was recently made by the 
	author in \cite{muehle14on}. In this article it was shown that every $\gamma$-compatible reflection order is a so-called
	recursive atom order of $\nc{W}$. This implies in particular that these lattices are CL-shellable. CL-shellability is a 
	kind of lexicographic shellability (possibly) different from EL-shellability. However, CL-shellability still implies all the
	topological and structural properties that EL-shellable posets enjoy. See \cite{bjorner83lexicographically} for
	more background on CL-shellability and recursive atom orders. It shall be remarked that the reasoning in \cite{muehle14on} is
	almost uniform, \ie the main argument does not depend on a case-by-case analysis, but it relies on two results for which
	no uniform proof is available.
\end{remark}

\section{EL-Shellability of $m$-Divisible Noncrossing Partitions}
  \label{sec:shellability_generalized}
In this section, we prove Theorem~\ref{thm:generalized_shellability_all}. To do so, we recall briefly the main idea 
of the proof of \cite{armstrong09generalized}*{Theorem~3.7.2}, which proves the analogous statement for real reflection
groups. So, for now assume that $W$ is finite a real reflection group, and let $\gamma\in W$ be a Coxeter element. Let 
$T_{\gamma}$ denote the set of all reflections of $W$ that are contained in $\nc{W}(\gamma)$, and fix a 
total order $\prec$ of $T_{\gamma}$ such that $\lambda_{\gamma}$ is an EL-labeling. Write 
$T_{\gamma}=\{t_{1},t_{2},\ldots,t_{N}\}$, where $t_{i}\prec t_{j}$ for $1\leq i<j\leq N$. First we need an 
EL-labeling for $\nc{W^{m}}$, where $W^{m}$ is the $m$-fold direct product of $W$ with itself. For $i\in\{1,2,\ldots,m\}$ 
and $j\in\{1,2,\ldots,N\}$, define a vector 
$\mathbf{t}_{i,j}=(\varepsilon,\varepsilon,\ldots,\varepsilon,t_{j},\varepsilon,\varepsilon,\ldots,\varepsilon)^{\tp}$, 
where $t_{j}\in T_{\gamma}$ appears in the $i$-th entry and $\varepsilon\in W$ denotes the identity. Define the set
$T_{\gamma}^{m}=\{\mathbf{t}_{i,j}\mid 1\leq i\leq m,1\leq j\leq N\}$, and consider the edge-labeling
\begin{multline*}
	\lambda_{\gamma}^{m}:\mathcal{E}(\nc{W^{m}})\to T_{\gamma}^{m},\\
	  \bigl((u_{1},u_{2},\ldots,u_{m}),(v_{1},v_{2},\ldots,v_{m})\bigr)\mapsto
	  \bigl(\lambda_{\gamma}(u_{1},v_{1}),\lambda_{\gamma}(u_{2},v_{2}),\ldots,\lambda_{\gamma}(u_{m},v_{m})\bigr),
\end{multline*}
where we use the additional convention that $\lambda_{\gamma}(w,w)=\varepsilon$ for $w\in\nc{W}$. Hence $T_{\gamma}^{m}$ is the
set of edge-labels of $\nc{W^{m}}$ with respect to $\lambda_{\gamma}^{m}$. Consider the total order $\prec_{m}$ of 
$T_{\gamma}^{m}$ given by
\begin{multline*}
	\mathbf{t}_{1,1}\prec_{m}\mathbf{t}_{1,2}\prec_{m}\cdots\prec_{m}\mathbf{t}_{1,N}\prec_{m}\mathbf{t}_{2,1}\\
	  \prec_{m}\mathbf{t}_{2,2}\prec_{m}\cdots\prec_{m}\mathbf{t}_{2,N}\prec_{m}\mathbf{t}_{3,1}\prec_{m}
	  \cdots\prec_{m}\mathbf{t}_{m,N}.
\end{multline*}
Then Theorem~\ref{thm:direct_product_shellable} implies that $\lambda_{\gamma}^{m}$ is an EL-labeling of $\nc{W^{m}}$. 

Lemma~3.4.3 in \cite{armstrong09generalized} implies that $\gnc{W}{m}$ is an order filter in $\nc{W^{m}}$. Thus
$\lambda_{\gamma}^{m}$ restricts to an edge-labeling of $\gnc{W}{m}$. Recall that $\bnc{W}{m}=\gnc{W}{m}\cup\{\hat{0}\}$ is the 
lattice that arises from $\gnc{W}{m}$ by adding a least element $\hat{0}$. Armstrong and Thomas introduce an abstract 
symbol $\delta$, and define an edge-labeling $\lambda_{\gamma}^{(m)}$ of $\bnc{W}{m}$ as follows: let 
$T_{\gamma}^{(m)}=T_{\gamma}^{m}\cup\{\delta\}$, and define
\begin{displaymath}
	\lambda_{\gamma}^{(m)}:\mathcal{E}\bigl(\bnc{W}{m}\bigr)\to T_{\gamma}^{(m)},\quad
	  (u,v)\mapsto\begin{cases}\delta, & \text{if}\;u=\hat{0},\\ \lambda_{\gamma}^{m}(v,u), & \text{otherwise}.\end{cases}
\end{displaymath}

Subsequently, they show that the total order $\prec_{(m)}$ of $T_{\gamma}^{(m)}$ given by 
\begin{multline}\label{eq:generalized_reflection_order}
	\mathbf{t}_{m,N}\prec_{(m)}\mathbf{t}_{m,N-1}\prec_{(m)}\cdots\prec_{(m)}\mathbf{t}_{m,1}\prec_{(m)}\mathbf{t}_{m-1,N}
		\prec_{(m)}\mathbf{t}_{m-1,N-1}\\
	\prec_{(m)}\cdots\prec_{(m)}\mathbf{t}_{2,1}\prec_{(m)}\delta\prec_{(m)}\mathbf{t}_{1,N}\prec_{(m)}\mathbf{t}_{1,N-1}
		\prec_{(m)}\cdots\prec_{(m)}\mathbf{t}_{1,1},
\end{multline}
makes $\lambda_{\gamma}^{(m)}$ an EL-labeling of $\bnc{W}{m}$. (Note that Armstrong and Thomas originally considered
the dual lattice of $\bnc{W}{m}$.)

\begin{proof}[Proof of Theorem~\ref{thm:generalized_shellability_all}]
	We observe that the construction described in the beginning of this section uses only structural properties of
	the poset $\gnc{W}{m}$ which can be generalized straightforwardly from real reflection groups to complex
	reflection groups. Theorem~\ref{thm:shellability_all} implies that for every well-generated complex 
	reflection group $W$ and for every Coxeter element $\gamma\in W$, we can find a total order on $T_{\gamma}$
	such that $\lambda_{\gamma}$ is an EL-labeling. Thus we can construct an EL-labeling $\lambda_{\gamma}^{(m)}$ of 
	$\bnc{W}{m}$ analogously to the construction of Armstrong and Thomas. 
\end{proof}

\begin{example}
  \label{ex:g553_2}
	Let us again consider the group $G(5,5,3)$, and identify the reflections in $T_{\gamma}$ with their position 
	in \eqref{eq:order_g553}, and consider the set $T_{\gamma}^{(2)}$. The total order 
	of $T_{\gamma}^{(2)}$ according to \eqref{eq:generalized_reflection_order} is
	\begin{displaymath}\begin{aligned}
		& (\varepsilon,12) && \prec_{(2)}(\varepsilon,11) && \prec_{(2)}(\varepsilon,10) && \prec_{(2)}(\varepsilon,9)
			&& \prec_{(2)}(\varepsilon,8) && \prec_{(2)}(\varepsilon,7) \\
		& && \prec_{(2)}(\varepsilon,6) && \prec_{(2)}(\varepsilon,5) && \prec_{(2)}(\varepsilon,4) 
			&& \prec_{(2)}(\varepsilon,3) && \prec_{(2)}(\varepsilon,2) \\
		& && \prec_{(2)}(\varepsilon,1) && \prec_{(2)}\delta && \prec_{(2)}(12,\varepsilon) && \prec_{(2)}(11,\varepsilon) 
			&& \prec_{(2)}(10,\varepsilon) \\
		& && \prec_{(2)}(9,\varepsilon) && \prec_{(2)}(8,\varepsilon) && \prec_{(2)}(7,\varepsilon) 
			&& \prec_{(2)}(6,\varepsilon) && \prec_{(2)}(5,\varepsilon)\\
		& && \prec_{(2)}(4,\varepsilon) && \prec_{(2)}(3,\varepsilon) && \prec_{(2)}(2,\varepsilon)
		    && \prec_{(2)}(1,\varepsilon).
	\end{aligned}\end{displaymath}
	The tuple $(\varepsilon,8)$, for instance, represents the tuple $\Bigl(\varepsilon,\colref{2}{3}{0}\Bigr)$.
	Figures~\ref{fig:nc2_g553_1} and \ref{fig:nc2_g553_2} display two intervals of $\bnc{G(5,5,3)}{2}$ with the
	EL-labeling $\lambda^{(2)}$. The nodes in each of these lattices are labeled by tuples which correspond to 
	$2$-divisible noncrossing partitions of $G(5,5,3)$ analogously to the labeling of the nodes of $\nc{G(5,5,3)}$
	in Figure~\ref{fig:nc_g553}. In each figure, the unique increasing maximal chain is indicated by thick edges. 
  
	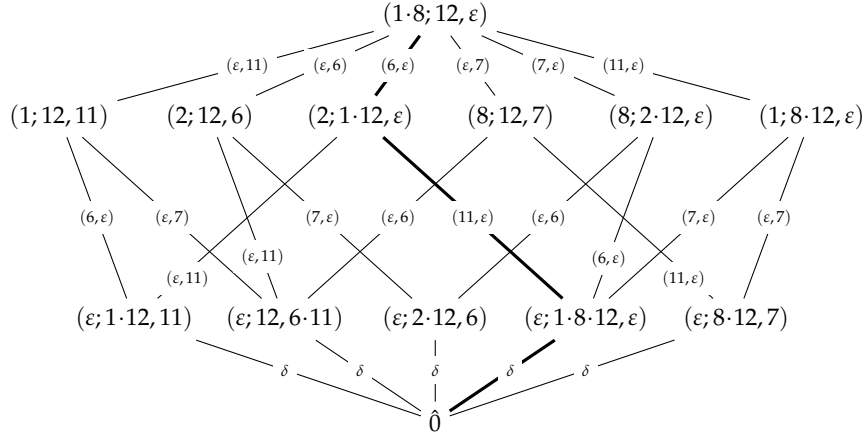
\begin{figure}
		\centering
		\begin{tikzpicture}\small
			\def\x{2};
			\def\y{2.7};
			\draw(3.5*\x,.5*\y) node(v0){$\hat{0}$};
			\draw(1.5*\x,1*\y) node(v1){$(\varepsilon;1\!\cdot\!12,11)$};
			\draw(2.5*\x,1*\y) node(v2){$(\varepsilon;12,6\!\cdot\!11)$};
			\draw(3.5*\x,1*\y) node(v3){$(\varepsilon;2\!\cdot\!12,6)$};
			\draw(4.5*\x,1*\y) node(v4){$(\varepsilon;1\!\cdot\!8\!\cdot\!12,\varepsilon)$};
			\draw(5.5*\x,1*\y) node(v5){$(\varepsilon;8\!\cdot\!12,7)$};
			\draw(1*\x,2*\y) node(v6){$(1;12,11)$};
			\draw(2*\x,2*\y) node(v7){$(2;12,6)$};
			\draw(3*\x,2*\y) node(v8){$(2;1\!\cdot\!12,\varepsilon)$};
			\draw(4*\x,2*\y) node(v9){$(8;12,7)$};
			\draw(5*\x,2*\y) node(v10){$(8;2\!\cdot\!12,\varepsilon)$};
			\draw(6*\x,2*\y) node(v11){$(1;8\!\cdot\!12,\varepsilon)$};
			\draw(3.5*\x,2.5*\y) node(v12){$(1\!\cdot\!8;12,\varepsilon)$};
			\draw(v0) -- (v1) node[fill=white] at(2.5*\x,.75*\y){\tiny $\delta$};
			\draw(v0) -- (v2) node[fill=white] at(3*\x,.75*\y){\tiny $\delta$};
			\draw(v0) -- (v3) node[fill=white] at(3.5*\x,.75*\y){\tiny $\delta$};
			\draw[very thick](v0) -- (v4) node[fill=white] at(4*\x,.75*\y){\tiny $\delta$};
			\draw(v0) -- (v5) node[fill=white] at(4.5*\x,.75*\y){\tiny $\delta$};
			\draw(v1) -- (v6) node[fill=white] at(1.25*\x,1.5*\y){\tiny $(6,\varepsilon)$};
			\draw(v1) -- (v8) node[fill=white] at(1.85*\x,1.2*\y){\tiny $(\varepsilon,11)$};
			\draw(v2) -- (v6) node[fill=white] at(1.75*\x,1.5*\y){\tiny $(\varepsilon,7)$};
			\draw(v2) -- (v7) node[fill=white] at(2.35*\x,1.31*\y){\tiny $(\varepsilon,11)$};
			\draw(v2) -- (v9) node[fill=white] at(3.25*\x,1.5*\y){\tiny $(\varepsilon,6)$};
			\draw(v3) -- (v7) node[fill=white] at(2.75*\x,1.5*\y){\tiny $(7,\varepsilon)$};
			\draw(v3) -- (v10) node[fill=white] at(4.25*\x,1.5*\y){\tiny $(\varepsilon,6)$};
			\draw[very thick](v4) -- (v8) node[fill=white] at(3.75*\x,1.5*\y){\tiny $(11,\varepsilon)$};
			\draw(v4) -- (v10) node[fill=white] at(4.65*\x,1.3*\y){\tiny $(6,\varepsilon)$};
			\draw(v4) -- (v11) node[fill=white] at(5.25*\x,1.5*\y){\tiny $(7,\varepsilon)$};
			\draw(v5) -- (v9) node[fill=white] at(5.15*\x,1.2*\y){\tiny $(11,\varepsilon)$};
			\draw(v5) -- (v11) node[fill=white] at(5.75*\x,1.5*\y){\tiny $(\varepsilon,7)$};
			\draw(v6) -- (v12) node[fill=white] at(2.25*\x,2.25*\y){\tiny $(\varepsilon,11)$};
			\draw(v7) -- (v12) node[fill=white] at(2.8*\x,2.25*\y){\tiny $(\varepsilon,6)$};
			\draw[very thick](v8) -- (v12) node[fill=white] at(3.25*\x,2.25*\y){\tiny $(6,\varepsilon)$};
			\draw(v9) -- (v12) node[fill=white] at(3.75*\x,2.25*\y){\tiny $(\varepsilon,7)$};
			\draw(v10) -- (v12) node[fill=white] at(4.25*\x,2.25*\y){\tiny $(7,\varepsilon)$};
			\draw(v11) -- (v12) node[fill=white] at(4.75*\x,2.25*\y){\tiny $(11,\varepsilon)$};
		\end{tikzpicture}
		\caption{An interval in $\bnc{G(5,5,3)}{2}$ with the labeling $\lambda^{(2)}$. The labels are 
		  explained in Example~\ref{ex:g553_1}.}
		\label{fig:nc2_g553_1}
	\end{figure}

	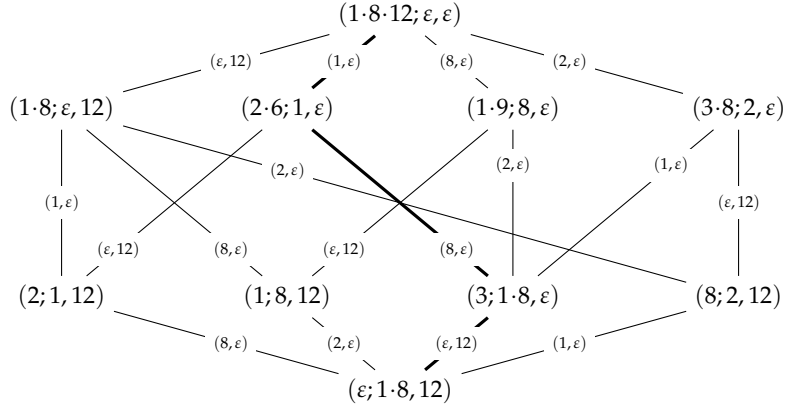
\begin{figure}
		\centering
		\begin{tikzpicture}\small
			\def\x{3};
			\def\y{2.5};
			\draw(2.5*\x,.5*\y) node(v0){$(\varepsilon;1\!\cdot\!8,12)$};
			\draw(1*\x,1*\y) node(v1){$(2;1,12)$};
			\draw(2*\x,1*\y) node(v2){$(1;8,12)$};
			\draw(3*\x,1*\y) node(v3){$(3;1\!\cdot\!8,\varepsilon)$};
			\draw(4*\x,1*\y) node(v4){$(8;2,12)$};
			\draw(1*\x,2*\y) node(v5){$(1\!\cdot\!8;\varepsilon,12)$};
			\draw(2*\x,2*\y) node(v6){$(2\!\cdot\!6;1,\varepsilon)$};
			\draw(3*\x,2*\y) node(v7){$(1\!\cdot\!9;8,\varepsilon)$};
			\draw(4*\x,2*\y) node(v8){$(3\!\cdot\!8;2,\varepsilon)$};
			\draw(2.5*\x,2.5*\y) node(v9){$(1\!\cdot\!8\!\cdot\!12;\varepsilon,\varepsilon)$};
			\draw(v0) -- (v1) node[fill=white] at(1.75*\x,.75*\y){\tiny $(8,\varepsilon)$};
			\draw(v0) -- (v2) node[fill=white] at(2.25*\x,.75*\y){\tiny $(2,\varepsilon)$};
			\draw[very thick](v0) -- (v3) node[fill=white] at(2.75*\x,.75*\y){\tiny $(\varepsilon,12)$};
			\draw(v0) -- (v4) node[fill=white] at(3.25*\x,.75*\y){\tiny $(1,\varepsilon)$};
			\draw(v1) -- (v5) node[fill=white] at(1*\x,1.5*\y){\tiny $(1,\varepsilon)$};
			\draw(v1) -- (v6) node[fill=white] at(1.25*\x,1.25*\y){\tiny $(\varepsilon,12)$};
			\draw(v2) -- (v5) node[fill=white] at(1.75*\x,1.25*\y){\tiny $(8,\varepsilon)$};
			\draw(v2) -- (v7) node[fill=white] at(2.25*\x,1.25*\y){\tiny $(\varepsilon,12)$};
			\draw[very thick](v3) -- (v6) node[fill=white] at(2.75*\x,1.25*\y){\tiny $(8,\varepsilon)$};
			\draw(v3) -- (v7) node[fill=white] at(3*\x,1.7*\y){\tiny $(2,\varepsilon)$};
			\draw(v3) -- (v8) node[fill=white] at(3.7*\x,1.7*\y){\tiny $(1,\varepsilon)$};
			\draw(v4) -- (v5) node[fill=white] at(2*\x,1.67*\y){\tiny $(2,\varepsilon)$};
			\draw(v4) -- (v8) node[fill=white] at(4*\x,1.5*\y){\tiny $(\varepsilon,12)$};
			\draw(v5) -- (v9) node[fill=white] at(1.75*\x,2.25*\y){\tiny $(\varepsilon,12)$};
			\draw[very thick](v6) -- (v9) node[fill=white] at(2.25*\x,2.25*\y){\tiny $(1,\varepsilon)$};
			\draw(v7) -- (v9) node[fill=white] at(2.75*\x,2.25*\y){\tiny $(8,\varepsilon)$};
			\draw(v8) -- (v9) node[fill=white] at(3.25*\x,2.25*\y){\tiny $(2,\varepsilon)$};
		\end{tikzpicture}
		\caption{Another interval in $\bnc{G(5,5,3)}{2}$ with the labeling $\lambda^{(2)}$. The labels are 
		  explained in Example~\ref{ex:g553_1}.}
		\label{fig:nc2_g553_2}
	\end{figure}
\end{example}
  
\section{Applications}
  \label{sec:application}
EL-shellability of a partially ordered set implies a certain structure of the associated order complex. In the present
case, this structure was already conjectured in \cite{armstrong09euler} and can now be proven. Recall that for every
positive integer $m$, the Fu{\ss}-Catalan numbers $\text{Cat}^{(m)}(W)$, see \eqref{eq:fuss_catalan}, count the 
$m$-divisible noncrossing partitions associated with a well-generated complex reflection group $W$. 
\begin{corollary}
  \label{cor:homotopy_truncated_order_complex}
	Let $W$ be a well-generated complex reflection group of rank $n$, and let $m$ be a positive integer. The order
	complex of the poset $\gnc{W}{m}$ with maximal and minimal elements removed is homotopy equivalent to a wedge
	of $(-1)^{n}\Bigl(\mbox{Cat}^{(-m-1)}(W)-\mbox{Cat}^{(-m)}(W)\Bigr)$-many $(n-2)$-spheres.
\end{corollary}
\begin{proof}
	Removing maximal and minimal elements from $\gnc{W}{m}$ yields a rank-selected subposet of $\gnc{W}{m}$. 
	Theorem~\ref{thm:generalized_shellability_all} and \cite{bjorner80shellable}*{Theorem~4.1} imply that this
	truncated poset is shellable. Hence the order complex associated with $\gnc{W}{m}$ with maximal and minimal 
	elements removed is also shellable. The result follows then with \cite{bjorner96shellable}*{Theorem~4.1} 
	and \cite{armstrong09euler}*{Theorem~9}.
\end{proof}

The previous result has consequences for the M{\"o}bius function of $\gnc{W}{m}$. 

\begin{corollary}
  \label{cor:mobius_function}
	Let $W$ be a well-generated complex reflection group of rank $n$, and let $\gamma\in W$ be a Coxeter element.
	Denote by $M$ the set of minimal elements of $\gnc{W}{m}(\gamma)$. Consider the lattice
	$\Bigl(\gnc{W}{m}(\gamma)\setminus M\Bigr)\cup\{\hat{0}\}$ that arises from $\gnc{W}{m}(\gamma)\setminus M$
	by adding a least element $\hat{0}$. For all positive integers $m$, we have 
	\begin{displaymath}
		\mu(\hat{0},\gamma)=\mbox{Cat}^{(-m-1)}(W)-\mbox{Cat}^{(-m)}(W),
	\end{displaymath}
	where $\mu$ denotes the corresponding M{\"o}bius function.
\end{corollary}
\begin{proof}
	Theorem~\ref{thm:generalized_shellability_all} implies that there exists an EL-labeling for 
	$\gnc{W}{m}(\gamma)\cup\{\hat{0}\}$ for any well-generated complex reflection group $W$ and every
	positive integer $m$. Hence the proof of this corollary works analogously to the proof of 
	\cite{tomie09mobius}*{Theorem~1.1}.
\end{proof}
  
\begin{example}
  \label{ex:g553_3}
	Let us finish the running example of $G(5,5,3)$. For $m=1$, we have 
	\begin{displaymath}
		\mbox{Cat}^{(-1)}\bigl(G(5,5,3)\bigr)=0\quad\mbox{and}\quad\mbox{Cat}^{(-2)}\bigl(G(5,5,3)\bigr)=-17
	\end{displaymath}
	
	Thus according to Corollary~\ref{cor:homotopy_truncated_order_complex}, the truncated order complex of
	$\nc{G(5,5,3)}$ is homotopy equivalent to a wedge of $17$ one-dimensional spheres, and it follows with
	\cite{bjorner96shellable}*{Theorem~5.9(i)} that there must be $17$ falling maximal chains in $\nc{G(5,5,3)}$. 
	This can be checked easily by inspecting Figure~\ref{fig:nc_g553}. 
	
	According to Corollary~\ref{cor:mobius_function}, the M{\"o}bius function of $\nc{G(5,5,3)}$ must satisfy
	$\mu(\varepsilon,\gamma)=-17$, which can again be checked by inspecting Figure~\ref{fig:nc_g553}.
\end{example}  

\section*{Acknowledgements}
  \label{sec:acknowledgements}
The author would like to thank the anonymous referees for their careful reading, and for their valuable suggestions
on presentation and content of the paper. Moreover, the author wants to thank Christian Krattenthaler for inspiring 
and helpful discussions on this topic.

\appendix

\section{The Proof of Lemma~\ref{lem:gddn_rank_2}}
  \label{app:nc_gddn_details_1}
\begin{lemma}
	Let $w\leq_{T}\gamma$ with $\ell_{T}(w)=2$. There exists a unique increasing reduced $T$-word of $w$ with respect to the restriction of 
	$\prec_{\gamma}$ to the reflections in $T_{\gamma}\cap[\varepsilon,w]$.
\end{lemma}
\begin{proof}
	Let $w=t_{1}t_{2}$ for $t_{1},t_{2}\in T_{\gamma}$. If $t_{1}$ and $t_{2}$ commute, then $w=t_{1}t_{2}=t_{2}t_{1}$ are the only possible 
	reduced $T$-words of $w$. Since $\prec_{\gamma}$ is a total order there is nothing to show.  
	Suppose that $t_{1}$ and $t_{2}$ do not commute. With the help of Proposition~\ref{prop:gamma_reflections}, we can explicitly 
	determine the possible forms of $w$. Analogously to the proof of Proposition~\ref{prop:gamma_reflections}, we investigate the fixed 
	space of $w^{-1}\gamma$ to determine which of these possibilities can actually occur in $\nc{G(d,d,n)}(\gamma)$. Recall from 
	\eqref{eq:coxeter_action} that for an arbitrary vector $\vb=(v_{1},v_{2},\ldots,v_{n})^{\tp}\in\mathbb{C}^{n}$, we have
	\begin{displaymath}
		\vb' = \gamma\vb = \Bigl(\zeta_{d}v_{n-1},v_{1},v_{2},\ldots,v_{n-2},\zeta_{d}^{d-1}v_{n}\Bigr)^{\tp}.
	\end{displaymath}
	
	(i) Let $t_{1}=\colref{a}{b}{0},t_{2}=\colref{b}{c}{0}$, where $1\leq a<b<c<n$. We have 
	$w=\bigl(\!\bigl(\colint{a}{0}\;\colint{b}{0}\;\colint{c}{0}\bigr)\!\bigr)$, and thus
	\begin{displaymath}
		w^{-1}\vb' = \Bigl(\zeta_{d}v_{n-1},v_{1},\ldots,v_{b-1},\ldots,v_{c-1},\ldots,v_{a-1},\ldots,v_{n-2},\zeta_{d}^{d-1}v_{n}\Bigr)^{\tp},
	\end{displaymath}
	and it follows that $w\leq_{T}\gamma$. Hence 
	$w^{-1}=\bigl(\!\bigl(\colint{a}{0}\;\colint{c}{0}\;\colint{b}{0}\bigr)\!\bigr)\not\leq_{T}\gamma$. The reduced $T$-words of $w$ are
	\begin{align*}
		w & = \colref{a}{b}{0}\colref{b}{c}{0}\\ 
		& = \colref{b}{c}{0}\colref{a}{c}{0}\\
		& = \colref{a}{c}{0}\colref{a}{b}{0},
	\end{align*}
	and according to \eqref{eq:order_gddn} only $w=\colref{a}{b}{0}\colref{b}{c}{0}$ is increasing.

	(ii) Let $t_{1}=\colref{a}{b}{0},t_{2}=\colref{a}{c}{0}$, where $1\leq a<b<c<n$. We have 
	$w=\bigl(\!\bigl(\colint{a}{0}\;\colint{c}{0}\;\colint{b}{0}\bigr)\!\bigr)$, and this was already considered in (i).
	
	(iii) Let $t_{1}=\colref{a}{b}{0},t_{2}=\colref{a}{b}{d-1}$, where $1\leq a<b<n$. We have $w=\bigl[\colint{a}{0}\bigr]\bigl[\colint{b}{0}\bigr]_{d-1}$,
	and 
	\begin{displaymath}
		w^{-1}\vb' = \Bigl(\zeta_{d}v_{n-1},v_{1},\ldots,\zeta_{d}v_{a-1},\ldots,\zeta_{d}^{d-1}v_{b-1},\ldots,v_{n-2},\zeta_{d}^{d-1}v_{n}\Bigr)^{\tp},
	\end{displaymath}
	and it follows that $w\not\leq_{T}\gamma$. On the other hand we have 
	\begin{displaymath}
		w\vb' = \Bigl(\zeta_{d}v_{n-1},v_{1},\ldots,\zeta_{d}^{d-1}v_{a-1},\ldots,\zeta_{d}v_{b-1},\ldots,v_{n-2},\zeta_{d}^{d-1}v_{n}\Bigr)^{\tp},
	\end{displaymath}
	and it follows again that $w^{-1}\not\leq_{T}\gamma$. 
	
	(iv) Let $t_{1}=\colref{a}{b}{0},t_{2}=\colref{b}{c}{d-1}$, where $1\leq a<b<c<n$. We have 
	$w=\bigl(\!\bigl(\colint{a}{0}\;\colint{b}{0}\;\colint{c}{d-1}\bigr)\!\bigr)$, and 
	\begin{displaymath}
		w^{-1}\vb' = \left(\zeta_{d}v_{n-1},v_{1},\ldots,v_{b-1},\ldots,\zeta_{d}v_{c-1},\ldots,\zeta_{d}^{d-1}v_{a-1},\ldots,
			v_{n-2},\zeta_{d}^{d-1}v_{n}\right)^{\tp},
	\end{displaymath}
	and it follows that $w\leq_{T}\gamma$. Hence $w^{-1}=\bigl(\!\bigl(\colint{a}{0}\;\colint{c}{d-1}\;\colint{b}{0}\bigr)\!\bigr)\not\leq_{T}\gamma$. 
	The reduced $T$-words of $w$ are
	\begin{align*}
		w & = \colref{a}{b}{0}\colref{b}{c}{d-1}\\
		& = \colref{b}{c}{d-1}\colref{a}{c}{d-1}\\
		& = \colref{a}{c}{d-1}\colref{a}{b}{0},
	\end{align*}
	and according to \eqref{eq:order_gddn} only $w=\colref{a}{b}{0}\colref{b}{c}{d-1}$ is increasing.
	
	(v) Let $t_{1}=\colref{a}{b}{0},t_{2}=\colref{a}{c}{d-1}$, where $1\leq a<b<c<n$. We have 
	$w=\bigl(\!\bigl(\colint{a}{0}\;\colint{c}{d-1}\;\colint{b}{0}\bigr)\!\bigr)$, and this was already considered in (iv).
	
	(vi) Let $t_{1}=\colref{a}{b}{0},t_{2}=\colref{b}{n}{s}$, where $1\leq a<b<n$ and $0\leq s<d$. We have 
	$w=\bigl(\!\bigl(\colint{a}{0}\;\colint{b}{0}\;\colint{n}{s}\bigr)\!\bigr)$, and
	\begin{displaymath}
		w^{-1}\vb' = \Bigl(\zeta_{d}v_{n-1},v_{1},\ldots,v_{b-1},\ldots,\zeta_{d}^{d-1-s}v_{n},\ldots,v_{n-2},\zeta_{d}^{s}v_{a-1}\Bigr)^{\tp},
	\end{displaymath}
	and it follows that $w\leq_{T}\gamma$. Hence $w^{-1}=\bigl(\!\bigl(\colint{a}{0}\;\colint{n}{s}\;\colint{b}{0}\bigr)\!\bigr)\not\leq_{T}\gamma$. 
	The reduced $T$-words of $w$ are 
	\begin{align*}
		w & = \colref{a}{b}{0}\colref{b}{n}{s}\\
		& = \colref{b}{n}{s}\colref{a}{n}{s}\\
		& = \colref{a}{n}{s}\colref{a}{b}{0},
	\end{align*}
	and according to \eqref{eq:order_gddn} only $w=\colref{a}{b}{0}\colref{b}{n}{s}$ is increasing.

	(vii) Let $t_{1}=\colref{a}{b}{0},t_{2}=\colref{a}{n}{s}$, where $1\leq a<b<n$ and $0\leq s<d$. We have 
	$w=\bigl(\!\bigl(\colint{a}{0}\;\colint{n}{s}\;\colint{b}{0}\bigr)\!\bigr)$, and this was already considered in (vi).
	
	(viii) Let $t_{1}=\colref{a}{c}{0},t_{2}=\colref{b}{c}{0}$, where $1\leq a<b<c<n$. We have 
	$w=\bigl(\!\bigl(\colint{a}{0}\;\colint{c}{0}\;\colint{b}{0}\bigr)\!\bigr)$, and 
	\begin{displaymath}
		w^{-1}\vb' = \Bigl(\zeta_{d}v_{n-1},v_{1},\ldots,v_{c-1},\ldots,v_{a-1},\ldots,v_{b-1},\ldots,v_{n-2},\zeta_{d}^{d-1}v_{n}\Bigr)^{\tp},
	\end{displaymath}
	and it follows that $w\not\leq_{T}\gamma$. On the other hand $w^{-1}=\bigl(\!\bigl(\colint{a}{0}\;\colint{b}{0}\;\colint{c}{0}\bigr)\!\bigr)$ was 
	considered in (i).
	
	(ix) Let $t_{1}=\colref{a}{c}{0},t_{2}=\colref{a}{b}{d-1}$, where $1\leq a<b<c<n$. We have 
	$w=\bigl(\!\bigl(\colint{a}{0}\;\colint{b}{d-1}\;\colint{c}{0}\bigr)\!\bigr)$, and 
	\begin{displaymath}
		w^{-1}\vb' = \Bigl(\zeta_{d}v_{n-1},v_{1},\ldots,\zeta_{d}v_{b-1},\ldots,\zeta_{d}^{d-1}v_{c-1},\ldots,v_{a-1},\ldots,
			v_{n-2},\zeta_{d}^{d-1}v_{n}\Bigr)^{\tp},
	\end{displaymath}
	and it follows that $w\not\leq_{T}\gamma$. On the other hand $w^{-1}=\bigl(\!\bigl(\colint{a}{0}\;\colint{c}{0}\;\colint{b}{d-1}\bigr)\!\bigr)$, and
	\begin{displaymath}
		w\vb' = \Bigl(\zeta_{d}v_{n-1},v_{1},\ldots,v_{c-1},\ldots,\zeta_{d}^{d-1}v_{a-1},\ldots,\zeta_{d}v_{b-1},\ldots,
			v_{n-2},\zeta_{d}^{d-1}v_{n}\Bigr)^{\tp},
	\end{displaymath}
	and it follows again that $w\not\leq_{T}\gamma$.
	
	(x) Let $t_{1}=\colref{a}{c}{0},t_{2}=\colref{b}{c}{d-1}$, where $1\leq a<b<c<n$. We have 
	$w=\bigl(\!\bigl(\colint{a}{0}\;\colint{c}{0}\;\colint{b}{1}\bigr)\!\bigr)$, and 
	\begin{displaymath}
		w^{-1}\vb' = \Bigl(\zeta_{d}v_{n-1},v_{1},\ldots,v_{c-1},\ldots,\zeta_{d}v_{a-1},\ldots,\zeta_{d}^{d-1}v_{b-1},\ldots,
			v_{n-2},\zeta_{d}^{d-1}v_{n}\Bigr)^{\tp},
	\end{displaymath}
	and it follows that $w\not\leq_{T}\gamma$. On the other hand $w^{-1}=\bigl(\!\bigl(\colint{a}{0}\;\colint{b}{1}\;\colint{c}{0}\bigr)\!\bigr)$, and 
	\begin{displaymath}
		w\vb' = \Bigl(\zeta_{d}v_{n-1},v_{1},\ldots,\zeta_{d}^{d-1}v_{b-1},\ldots,\zeta_{d}v_{c-1},\ldots,v_{a-1},\ldots,
			v_{n-2},\zeta_{d}^{d-1}v_{n}\Bigr)^{\tp},
	\end{displaymath}
	and it follows again that $w^{-1}\not\leq_{T}\gamma$. 

	(xi) Let $t_{1}=\colref{b}{c}{0},t_{2}=\colref{a}{c}{d-1}$, where $1\leq a<b<c<n$. We have 
	$w=\bigl(\!\bigl(\colint{a}{0}\;\colint{b}{d-1}\;\colint{c}{d-1}\bigr)\!\bigr)$, and
	\begin{displaymath}
		w\vb' = \Bigl(\zeta_{d}v_{n-1},v_{1},\ldots,\zeta_{d}v_{b-1},\ldots,v_{c-1},\ldots,\zeta_{d}^{d-1}v_{a-1},\ldots,
			v_{n-2},\zeta_{d}^{d-1}v_{n}\Bigr)^{\tp},
	\end{displaymath}
	and it follows that $w\leq_{T}\gamma$. Hence 
	$w^{-1}=\bigl(\!\bigl(\colint{a}{0}\;\colint{c}{d-1}\;\colint{b}{d-1}\bigr)\!\bigr)\not\leq_{T}\gamma$. The reduced $T$-words of $w$ are 
	\begin{align*}
		w & = \colref{a}{b}{d-1}\colref{b}{c}{0}\\
		& = \colref{b}{c}{0}\colref{a}{c}{d-1}\\
		& = \colref{a}{c}{d-1}\colref{a}{b}{d-1},
	\end{align*}
	and according to \eqref{eq:order_gddn} only $w=\colref{b}{c}{0}\colref{a}{c}{d-1}$ is increasing.
	
	(xii) Let $t_{1}=\colref{b}{c}{0},t_{2}=\colref{a}{b}{d-1}$, where $1\leq a<b<c<n$. We have 
	$w=\bigl(\!\bigl(\colint{a}{0}\;\colint{c}{d-1}\;\colint{b}{d-1}\bigr)\!\bigr)$, and this was already considered in (xi).

	(xiii) Let $t_{1}=\colref{a}{b}{d-1},t_{2}=\colref{b}{c}{d-1}$, where $1\leq a<b<c<n$. We have 
	$w=\bigl(\!\bigl(\colint{a}{0}\;\colint{b}{d-1}\;\colint{c}{d-2}\bigr)\!\bigr)$, and 
	\begin{displaymath}
		w^{-1}\vb' = \Bigl(\zeta_{d}v_{n-1},v_{1},\ldots,\zeta_{d}v_{b-1},\ldots,\zeta_{d}v_{c-1},\ldots,\zeta_{d}^{d-2}v_{a-1},\ldots,
			v_{n-2},\zeta_{d}^{d-1}v_{n}\Bigr)^{\tp},
	\end{displaymath}
	and it follows that $w\not\leq_{T}\gamma$. On the other hand
	$w^{-1}=\bigl(\!\bigl(\colint{a}{0}\;\colint{c}{d-2}\;\colint{b}{d-1}\bigr)\!\bigr)$, and
	\begin{displaymath}
		w\vb' = \Bigl(\zeta_{d}v_{n-1},v_{1},\ldots,\zeta_{d}^{d-2}v_{c-1},\ldots,\zeta^{d-1}_{d}v_{a-1},\ldots,\zeta_{d}^{d-1}v_{b-1},\ldots,
			v_{n-2},\zeta_{d}^{d-1}v_{n}\Bigr)^{\tp},
	\end{displaymath}
	and it follows again that $w^{-1}\not\leq_{T}\gamma$.
		
	(xiv) Let $t_{1}=\colref{a}{b}{d-1},t_{2}=\colref{a}{c}{d-1}$, where $1\leq a<b<c<n$. We have 
	$w=\bigl(\!\bigl(\colint{a}{0}\;\colint{c}{d-1}\;\colint{b}{d-1}\bigr)\!\bigr)$, and this was already considered in (xii). On the other hand
	we have $w^{-1}=\bigl(\!\bigl(\colint{a}{0}\;\colint{b}{d-1}\;\colint{c}{d-1}\bigr)\!\bigr)$, and this was already considered in (xi).
	
	(xv) Let $t_{1}=\colref{a}{b}{d-1},t_{2}=\colref{a}{n}{s}$, where $1\leq a<b<n$ and $0\leq s<d$. We have 
	$w=\bigl(\!\bigl(\colint{a}{0}\;\colint{n}{s}\;\colint{b}{d-1}\bigr)\!\bigr)$, and 
	\begin{displaymath}
		w^{-1}\vb' = \Bigl(\zeta_{d}v_{n-1},v_{1},\ldots,\zeta_{d}^{d-1-s}v_{n},\ldots,\zeta_{d}^{d-1}v_{a-1},\ldots,
			v_{n-2},\zeta_{d}^{s+1}v_{b-1}\Bigr)^{\tp},
	\end{displaymath}
	and it follows that $w\leq_{T}\gamma$. Hence $w^{-1}=\bigl(\!\bigl(\colint{a}{0}\;\colint{b}{d-1}\;\colint{n}{s}\bigr)\!\bigr)\not\leq_{T}\gamma$. 
	The reduced $T$-words of $w$ are 
	\begin{align*}
		w & = \colref{a}{n}{s}\colref{b}{n}{s+1}\\
		& = \colref{b}{n}{s+1}\colref{a}{b}{d-1}\\
		& = \colref{a}{b}{d-1}\colref{a}{n}{s},
	\end{align*}
	and according to \eqref{eq:order_gddn} only $w=\colref{a}{n}{s}\colref{b}{n}{s+1}$ is increasing.

	(xvi) Let $t_{1}=\colref{a}{b}{d-1},t_{2}=\colref{b}{n}{s}$, where $1\leq a<b<c<n$. We have 
	$w=\bigl(\!\bigl(\colint{a}{0}\;\colint{b}{d-1}\;\colint{n}{s-1}\bigr)\!\bigr)$, and this was already considered in (xv). 
	
	(xvii) Let $t_{1}=\colref{a}{c}{d-1},t_{2}=\colref{b}{c}{d-1}$, where $1\leq a<b<c<n$. We have 
	$w=\bigl(\!\bigl(\colint{a}{0}\;\colint{c}{d-1}\;\colint{b}{0}\bigr)\!\bigr)$, and this was already considered in (iv).
	
	(xviii) Let $t_{1}=\colref{a}{n}{s},t_{2}=\colref{a}{n}{t}$, where $1\leq a<n$ and $0\leq s,t<d$ with $t\neq s$. We have 
	$w=\bigl[\colint{a}{0}\bigr]_{t-s}\bigl[\colint{n}{0}\bigr]_{s-t}$, and
	\begin{displaymath}
		w^{-1}\vb' = \Bigl(\zeta_{d}v_{n-1},v_{1},\ldots,\zeta_{d}^{s-t}v_{a-1},\ldots,v_{n-2},\zeta_{d}^{t-1-s}v_{n}\Bigr)^{\tp},
	\end{displaymath}
	and it follows that $w\leq_{T}\gamma$ if and only if $t=s+1$. In this case the reduced $T$-words of $w$ are 
	\begin{align*}
		w & = \colref{a}{n}{s}\colref{a}{n}{s+1}\\
		& = \colref{a}{n}{s+1}\colref{a}{n}{s+2}\\ 
		&	= \colref{a}{n}{s+2}\colref{a}{n}{s+3}\\
		& = \quad\cdots\\
		& = \colref{a}{n}{s-1}\colref{a}{n}{s}.
	\end{align*}
	and according to \eqref{eq:order_gddn} only $w=\colref{a}{n}{0}\colref{a}{n}{1}$ is increasing.
	
	Thus the proof is complete.
\end{proof}

\section{The Proof of Proposition~\ref{prop:gddn_symmetric_irreducible}}
  \label{app:nc_gddn_details_2}
\begin{proposition}
	Let $w\leq_{T}\gamma$ such that the parabolic subgroup of $G(d,d,n)$, in which $w$ is a Coxeter element, is isomorphic to $G(1,1,n')$ for some 
	$n'\leq n$. Then $w$ is of one of the following three forms:
	\begin{enumerate}[(i)]
		\item $w=\bigl(\!\bigl(\colint{(a\!+\!1)}{0}\;\colint{(a\!+\!2)}{0}\;\ldots\;\colint{b}{0}\bigr)\!\bigr)$, where $1\leq a<b<n$,
		\item $w=\bigl(\!\bigl(\colint{1}{0}\;\colint{2}{0}\;\ldots\;\colint{a}{0}\;\colint{(b\!+\!1)}{d-1}\;\colint{(b\!+\!2)}{d-1}\;\ldots\;
			\colint{(n\!-\!1)}{d-1}\bigr)\!\bigr)$, where $1\leq a<b<n$, \quad or
		\item $w=\bigl(\!\bigl(\colint{1}{0}\;\colint{2}{0}\;\ldots\;\colint{a}{0}\;\colint{n}{s-1}\colint{(a\!+\!1)}{d-1}\;\colint{(a\!+\!2)}{d-1}\;
			\ldots\;\colint{(n\!-\!1)}{d-1}\bigr)\!\bigr)$, where $1\leq a<n$.
	\end{enumerate}
	Moreover, in each of these cases there exists a unique increasing reduced $T$-word of $w$ with respect to the restriction of $\prec_{\gamma}$ 
	to the reflections in $T_{\gamma}\cap[\varepsilon,w]$.
\end{proposition}
\begin{proof}
	The observation that $w$ can only be of the forms (i)--(iii) is a straightforward computation using 
	Proposition~\ref{prop:gamma_reflections}. For the second part of the proposition, we proceed by induction on $\ell_{T}(w)$. If 
	$\ell_{T}(w)=2$, then the claim follows from Lemma~\ref{lem:gddn_rank_2}. Suppose that $\ell_{T}(w)=k$, and suppose that the claim is true for all 
	suitable $w'$ with $\ell_{T}(w')<k$. 
	
	\smallskip
	
	(i) Let $w=\bigl(\!\bigl(\colint{(a\!+\!1)}{0}\;\colint{(a\!+\!2)}{0}\;\ldots\;\colint{b}{0}\bigr)\!\bigr)$, where $1\leq a<b<n$. Consider the 
	decomposition of $w$ according to \eqref{eq:factors_g11n}:
	\begin{displaymath}
		w = \colref{(a\!+\!1)}{(a\!+\!2)}{0}\colref{(a\!+\!2)}{(a\!+\!3)}{0}\cdots\colref{(b\!-\!1)}{b}{0}.
	\end{displaymath}
	We notice that this word is increasing with respect to \eqref{eq:order_gddn}, and the claim follows now analogously to the proof of 
	Lemma~\ref{lem:sym_shellable}. 
	
	\smallskip
	
	(ii) Let $w=\bigl(\!\bigl(\colint{1}{0}\;\colint{2}{0}\;\ldots\;\colint{a}{0}\;\colint{(b\!+\!1)}{d-1}\;\colint{(b\!+\!2)}{d-1}\;\ldots\;
	\colint{(n\!-\!1)}{d-1}\bigr)\!\bigr)$, where $1\leq a<b<n$. Again consider the decomposition of $w$ according to 
	\eqref{eq:factors_g11n}:
	\begin{multline*}
		w = \colref{1}{2}{0}\colref{2}{3}{0}\cdots\colref{(a\!-\!1)}{a}{0}\colref{a}{(b\!+\!1)}{d-1}\\
			\colref{(b\!+\!1)}{(b\!+\!2)}{0}\cdots\colref{(n\!-\!2)}{(n\!-\!1)}{0}.
	\end{multline*}
	We notice that this word is not increasing with respect to \eqref{eq:order_gddn}. However, repeated left-shifting yields
	\begin{multline}\label{eq:g11n_rising_2}
		w = \colref{1}{2}{0}\colref{2}{3}{0}\cdots\colref{(a\!-\!1)}{a}{0}\\
			\colref{(b\!+\!1)}{(b\!+\!2)}{0}\cdots\colref{(n\!-\!2)}{(n\!-\!1)}{0}\colref{a}{(n\!-\!1)}{d-1},
	\end{multline}
	and this word is increasing with respect to \eqref{eq:order_gddn}. We need to show that this is the only increasing reduced $T$-word of $w$.
	Suppose that $w=t_{1}t_{2}\cdots t_{k}$ is an increasing reduced $T$-word of $w$ that is different from \eqref{eq:g11n_rising_2}. Suppose that
	$i$ is the maximal index where this word differs from \eqref{eq:g11n_rising_2}. If $i<k$, then $t_{1}t_{2}\cdots t_{i}$ is a product of at 
	most two cycles of the form (i), and it follows that $t_{1}t_{2}\cdots t_{i}$ is increasing only if $t_{j}$ is the $j$-th factor in 
	\eqref{eq:g11n_rising_2} for all $j\in\{1,2,\ldots,i\}$, which is a contradiction. Now let $i=k$, and consider the word $w'=wt_{k}$. It follows by 
	induction hypothesis that the product of the first $k-1$ factors in \eqref{eq:g11n_rising_2} is the unique increasing reduced $T$-word of $w'$. 
	In view of Lemma~\ref{lem:shifting} and Proposition~\ref{prop:gamma_reflections} the reflection $t_{k}$ can only be of one of the 
	following four forms. \\
	(iia) Let $t_{k}=\colref{a}{(n\!-\!1)}{d-1}$. Then $t_{k}$ is the $k$-th factor in \eqref{eq:g11n_rising_2}, and we obtain a contradiction.\\
	(iib) Let $t_{k}=\colref{a}{c}{d-1}$, where $b+1\leq c<n-1$. We have 
	\begin{multline*}
		w' = \bigl(\!\bigl(\colint{1}{0}\;\colint{2}{0}\;\ldots\;\colint{a}{0}\,\colint{(c\!+\!1)}{d-1}\;\colint{(c\!+\!2)}{d-1}\;\ldots\\
			\colint{(n\!-\!1)}{d-1}\bigr)\!\bigr)\bigl(\!\bigl(\colint{(b\!+\!1)}{0}\;\colint{(b\!+\!2)}{0}\;\ldots\;\colint{c}{0}\bigr)\!\bigr).
	\end{multline*}
	Hence we can write $w'=w'_{1}w'_{2}$, where $w'_{1}$ is again of type (ii) and $w'_{2}$ is of type (i). In particular 
	$\ell_{T}(w'_{1}),\ell_{T}(w'_{2})<k$, so by induction hypothesis $w'_{1}$ and $w'_{2}$ possess a unique increasing reduced $T$-word, namely
	\begin{multline*}
		w'_{1} = \colref{1}{2}{0}\colref{2}{3}{0}\cdots\colref{(a\!-\!1)}{a}{0}\\
			\colref{(c\!+\!1)}{(c\!+\!2)}{0}\cdots\colref{(n\!-\!2)}{(n\!-\!1)}{0}\colref{a}{(n\!-\!1)}{d-1},
	\end{multline*}
	and 
	\begin{displaymath}
		w'_{2} = \colref{(b\!+\!1)}{(b\!+\!2)}{0}\colref{(b\!+\!2)}{(b\!+\!3)}{0}\cdots\colref{(c\!-\!1)}{c}{0}.
	\end{displaymath}
	Now we can quickly verify that
	\begin{multline*}
		w' = \colref{1}{2}{0}\cdots\colref{(a\!-\!1)}{a}{0}\colref{(b\!+\!1)}{(b\!+\!2)}{0}\cdots\colref{(c\!-\!1)}{c}{0}\\
			\colref{(c\!+\!1)}{(c\!+\!2)}{0}\cdots\colref{(n\!-\!2)}{(n\!-\!1)}{0}\colref{a}{(n\!-\!1)}{d-1},
	\end{multline*}
	is the unique increasing reduced $T$-word of $w'$ and hence has to correspond to $t_{1}t_{2}\cdots t_{k-1}$. (Indeed, first concatenate the increasing 
	words of $w'_{1}$ and $w'_{2}$, and observe that the resulting word is not increasing. Then shift the first factor, say $r$, of the 
	increasing word of $w'_{2}$ as far to the left as possible such that the resulting prefix, say $r_{1}r_{2}\cdots r_{l}$, is increasing, where 
	$r_{1}r_{2}\cdots r_{l-1}$ is a prefix of the increasing word of $w'_{1}$ and $r_{l}=r$. Then observe that shifting $r$ further to 
	the left yields a non-increasing prefix $r'_{1}r'_{2}\cdots r'_{l}$. Proceed analogously until you have reached the last factor of the increasing word 
	of $w'_{2}$.) However, we have for instance $\colref{a}{(n\!-\!1}{d-1}\succ_{\gamma}\colref{a}{c}{d-1}=t_{k}$, which contradicts the assumption 
	that $t_{1}t_{2}\cdots t_{k}$ is increasing.\\
	(iic) Let $t_{k}=\colref{c}{(c\!+\!1)}{0}$, where $b+1\leq c<n-1$. We have
	\begin{multline*}
		w' = \bigl(\!\bigl(\colint{1}{0}\;\ldots\;\colint{a}{0}\,\colint{(b\!+\!1)}{d-1}\;\colint{(b\!+\!2)}{d-1}\;\ldots\;
			\colint{c}{d-1}\\
			\colint{(c\!+\!2)}{d-1}\;\colint{(c\!+\!3)}{d-1}\;\ldots\;\colint{(n\!-\!1)}{d-1}\bigr)\!\bigr),
	\end{multline*}
	and $\ell_{T}(w')<k$. Moreover, $w'$ is again of type (ii), so by induction hypothesis there exists a unique increasing reduced $T$-word of $w'$, 
	namely
	\begin{align*}
		w' & = \colref{1}{2}{0}\cdots\colref{(a\!-\!1)}{a}{0}\colref{(b\!+\!1)}{(b\!+\!2)}{0}\cdots\colref{(c\!-\!1)}{c}{0}\\
		& \kern1cm\colref{c}{c\!+\!2}{0}\colref{(c\!+\!2)}{(c\!+\!3)}{0}\cdots\colref{(n\!-\!2)}{(n\!-\!1)}{0}\\
		& \kern1cm\colref{a}{(n\!-\!1)}{d-1},
	\end{align*}
	and thus this word has to correspond to $t_{1}t_{2}\cdots t_{k-1}$. However, we have for instance 
	$\colref{a}{(n\!-\!1)}{d-1}\succ_{\gamma}\colref{c}{(c\!+\!1)}{0}=t_{k}$, which contradicts the assumption that $t_{1}t_{2}\cdots t_{k}$ is increasing.\\
	(iid) Let $t_{k}=\colref{c}{(c\!+\!1)}{0}$, where $1\leq c<a$. We have
	\begin{multline*}
		w' = \bigl(\!\bigl(\colint{1}{0}\;\ldots\;\colint{c}{0}\;\colint{(c+2)}{0}\;\colint{(c\!+\!3)}{0}\;\ldots\;\colint{a}{0}\;\colint{(b\!+\!1)}{d-1}\\
			\colint{(b\!+\!2)}{d-1}\;\ldots\;\colint{(n\!-\!1)}{d-1}\bigr)\!\bigr),
	\end{multline*}
	and $\ell_{T}(w')<k$. Moreover, $w'$ is again of type (ii), so by induction hypothesis there exists a unique increasing reduced $T$-word of $w'$, 
	namely
	\begin{align*}
		w' & = \colref{1}{2}{0}\cdots\colref{c}{(c\!+\!2)}{0}\colref{(c\!+\!2)}{(c\!+\!3)}{0}\cdots\colref{(a\!-\!1)}{a}{0}\\
		& \kern1cm\colref{(b\!+\!1)}{(b\!+\!2)}{0}\colref{(b\!+\!2)}{(b\!+\!3)}{0}\cdots\colref{(n\!-\!2)}{(n\!-\!1)}{0}\\
		& \kern1cm\colref{a}{(n\!-\!1)}{d-1},
	\end{align*}
	and thus this word has to correspond to $t_{1}t_{2}\cdots t_{k-1}$. However, we have for instance 
	$\colref{a}{(n\!-\!1)}{d-1}\succ_{\gamma}\colref{c}{(c\!+\!1)}{0}=t_{k}$, which contradicts the assumption that $t_{1}t_{2}\cdots t_{k}$ is increasing.\\
	Hence the reduced $T$-word of $w$ in \eqref{eq:g11n_rising_2} is the unique increasing reduced $T$-word.
	
	\smallskip
	
	(iii) Let $w=\bigl(\!\bigl(\colint{1}{0}\;\colint{2}{0}\;\ldots\;\colint{a}{0}\;\colint{n}{s-1}\colint{(a\!+\!1)}{d-1}\;\colint{(a\!+\!2)}{d-1}\;
	\ldots\;\colint{(n\!-\!1)}{d-1}\bigr)\!\bigr)$, where $1\leq a<n$. Again consider the decomposition of $w$ according to 
	\eqref{eq:factors_g11n}:
	\begin{multline*}
		w = \colref{1}{2}{0}\colref{2}{3}{0}\cdots\colref{(a\!-\!1)}{a}{0}\colref{a}{n}{s-1}\colref{(a\!+\!1)}{n}{s}\\
			\colref{(a\!+\!1)}{(a\!+\!2)}{0}\colref{(a\!+\!2)}{(a\!+\!3)}{0}\cdots\colref{(n\!-\!2)}{(n\!-\!1)}{0}.
	\end{multline*}
	We notice that this word is not increasing with respect to \eqref{eq:order_gddn}. However, repeated left-shifting yields
	\begin{align}\label{eq:g11n_rising_3}
		w & = \colref{1}{2}{0}\colref{2}{3}{0}\cdots\colref{(a\!-\!1)}{a}{0}\colref{(a\!+\!1)}{(a\!+\!2)}{0}\\
		& \kern1cm\colref{(a\!+\!2)}{(a\!+\!3)}{0}\cdots\colref{(n\!-\!2)}{(n\!-\!1)}{0}\colref{a}{n}{s-1}\nonumber\\
		& \kern1cm\colref{(n\!-\!1)}{n}{s}\nonumber,
	\end{align}
	and this word is increasing with respect to \eqref{eq:order_gddn}. We need to show that this is the only increasing reduced $T$-word of $w$.
	Again we suppose that $w=t_{1}t_{2}\cdots t_{k}$ is an increasing reduced $T$-word of $w$ that is different from \eqref{eq:g11n_rising_3}, and
	analogously to (ii) it suffices to investigate $w'=wt_{k}$. In view of Lemma~\ref{lem:shifting} and 
	Proposition~\ref{prop:gamma_reflections} the reflection $t_{k}$ can only be of one of the following five forms. \\
	(iiia) Let $t_{k}=\colref{(n\!-\!1)}{n}{s}$. Then $t_{k}$ is the $k$-th factor in \eqref{eq:g11n_rising_3}, and we obtain a contradiction.\\
	(iiib) Let $t_{k}=\colref{c}{n}{s}$, where $a+1\leq c<n-1$. We have 
	\begin{multline*}
		w' = \bigl(\!\bigl(\colint{1}{0}\;\colint{2}{0}\;\ldots\;\colint{a}{0}\;\colint{n}{s-1}\;\colint{(c\!+\!1)}{d-1}\;
			\colint{(c\!+\!2)}{d-1}\;\ldots\;\colint{(n\!-\!1)}{d-1}\bigr)\!\bigr)\\
			\bigl(\!\bigl(\colint{(a\!+\!1)}{0}\;\colint{(a\!+\!2)}{0}\;\ldots\;\colint{c}{0}\bigr)\!\bigr).
	\end{multline*}
	Hence we can write $w'=w'_{1}w'_{2}$, where $w'_{1}$ is again of type (iii) and $w'_{2}$ is of type (i). In particular 
	$\ell_{T}(w'_{1}),\ell_{T}(w'_{2})<k$, so by induction hypothesis we can find a unique increasing reduced $T$-word of $w'$ analogously to (iib), namely
	\begin{align*}
		w' & = \colref{1}{2}{0}\colref{2}{3}{0}\cdots\colref{(a\!-\!1)}{a}{0}\colref{(a\!+\!1)}{(a\!+\!2)}{0}\cdots\\
		& \kern1cm\colref{(c\!-\!1)}{c}{0}\colref{(c\!+\!1)}{(c\!+\!2)}{0}\cdots\colref{(n\!-\!2)}{(n\!-\!1)}{0}\\
		& \kern1cm\colref{a}{n}{s-1}\colref{(n\!-\!1)}{n}{s},
	\end{align*}
	and thus this word has to correspond to $t_{1}t_{2}\cdots t_{k-1}$. However, we have for instance 
	$\colref{(n\!-\!1)}{n}{s}\succ_{\gamma}\colref{c}{n}{s}=t_{k}$, which implies that there exists no increasing reduced $T$-word of $w$ in this 
	case.\\
	(iiic) Let $t_{k}=\colref{c}{n}{s-1}$, where $a\leq c<n-1$. We have 
	\begin{multline*}
		w' = \bigl(\!\bigl(\colint{1}{0}\;\colint{2}{0}\;\ldots\;\colint{a}{0}\;\colint{n}{s-1}\;\colint{(c\!+\!1)}{0}\;
			\colint{(c\!+\!2)}{0}\;\ldots\;\colint{(n\!-\!1)}{0}\bigr)\!\bigr)\\
			\bigl(\!\bigl(\colint{(a\!+\!1)}{0}\;\colint{(a\!+\!2)}{0}\;\ldots\;\colint{c}{0}\bigr)\!\bigr).
	\end{multline*}
	Hence we can write $w'=w'_{1}w'_{2}$, where $w'_{1}$ is again of type (iii) and $w'_{2}$ is of type (i). In particular 
	$\ell_{T}(w'_{1}),\ell_{T}(w'_{2})<k$, so by induction hypothesis we can find a unique increasing reduced $T$-word of $w'$ analogously to (iiib), namely
	\begin{align*}
		w' & = \colref{1}{2}{0}\colref{2}{3}{0}\cdots\colref{(a\!-\!1)}{a}{0}\colref{(a\!+\!1)}{(a\!+\!2)}{0}\cdots\\
		& \kern1cm\colref{(c\!-\!1)}{c}{0}\colref{(c\!+\!1)}{(c\!+\!2)}{0}\cdots\colref{(n\!-\!2)}{(n\!-\!1)}{0}\\
		& \kern1cm\colref{c}{n}{s-1}\colref{(n\!-\!1)}{n}{s-1},
	\end{align*}
	and thus this word has to correspond to $t_{1}t_{2}\cdots t_{k-1}$. However, we have for instance 
	$\colref{(n\!-\!1)}{n}{s-1}\succ_{\gamma}\colref{c}{n}{s-1}=t_{k}$, which contradicts the assumption that $t_{1}t_{2}\cdots t_{k}$ is increasing.\\
	(iiid) Let $t_{k}=\colref{c}{(c\!+\!1)}{0}$, where $a+1\leq c<n-1$. We have
	\begin{multline*}
		w' = \bigl(\!\bigl(\colint{1}{0}\;\ldots\;\colint{a}{0}\,\colint{n}{s-1}\;\colint{(a\!+\!1)}{d-1}\;\ldots\;\colint{c}{d-1}\;\colint{(c\!+\!2)}{d-1}\\
			\colint{(c\!+\!3)}{d-1}\;\ldots\;\colint{(n\!-\!1)}{d-1}\bigr)\!\bigr),
	\end{multline*}
	and $\ell_{T}(w')<k$. Moreover, $w'$ is again of type (iii), so by induction hypothesis there exists a unique increasing reduced $T$-word of 
	$w'$, namely
	\begin{align*}
		w' & = \colref{1}{2}{0}\cdots\colref{(a\!-\!1)}{a}{0}\colref{(a\!+\!1)}{(a\!+\!2)}{0}\cdots\colref{(c\!-\!1)}{c}{0}\\
		& \kern1cm\colref{c}{c\!+\!2}{0}\colref{(c\!+\!2)}{(c\!+\!3)}{0}\cdots\colref{(n\!-\!2)}{(n\!-\!1)}{0}\\
		& \kern1cm\colref{a}{(n\!-\!1)}{s-1}\colref{(n\!-\!1)}{n}{s},
	\end{align*}
	and thus this word has to correspond to $t_{1}t_{2}\cdots t_{k-1}$. However, we have for instance 
	$\colref{(n\!-\!1)}{n}{s}\succ_{\gamma}\colref{c}{(c\!+\!1)}{0}=t_{k}$, which contradicts the assumption that $t_{1}t_{2}\cdots t_{k}$ is increasing.\\
	(iiie) Let $t_{k}=\colref{c}{(c\!+\!1)}{0}$, where $1\leq c<a$. We have
	\begin{displaymath}
		w' = \bigl(\!\bigl(\colint{1}{0}\;\ldots\;\colint{c}{0}\;\colint{(c+2)}{0}\;\colint{(c\!+\!3)}{0}\;\ldots\;\colint{a}{0}\;
			\colint{n}{s-1}\;\colint{(a\!+\!1)}{d-1}\;\ldots\;\colint{(n\!-\!1)}{d-1}\bigr)\!\bigr),
	\end{displaymath}
	and $\ell_{T}(w')<k$. Moreover, $w'$ is again of type (iii), so by induction hypothesis there exists a unique increasing reduced $T$-word of 
	$w'$, namely
	\begin{align*}
		w' & = \colref{1}{2}{0}\cdots\colref{(c\!-\!1)}{c}{0}\colref{c}{(c\!+\!2)}{0}\colref{(c\!+\!2)}{(c\!+\!3)}{0}\cdots\\
		& \kern1cm\colref{(a\!-\!1)}{a}{0}\colref{(a\!+\!1)}{(a\!+\!2)}{0}\cdots\colref{(n\!-\!2)}{(n\!-\!1)}{0}\\
		& \kern1cm\colref{a}{n}{s-1}\colref{(n\!-\!1)}{n}{s},
	\end{align*}
	and thus this word has to correspond to $t_{1}t_{2}\cdots t_{k-1}$. However, we have for instance 
	$\colref{(n\!-\!1)}{n}{s}\succ_{\gamma}\colref{c}{(c\!+\!1)}{0}=t_{k}$, which contradicts the assumption that $t_{1}t_{2}\cdots t_{k}$ is increasing.\\
	Hence the reduced $T$-word of $w$ in \eqref{eq:g11n_rising_3} is the unique increasing reduced $T$-word, and the proof is complete.
\end{proof}

\section{The Proof of Corollary~\ref{cor:gddn_symmetric_reducible}}
  \label{app:nc_gddn_details_3}
\begin{corollary}
	Let $w\leq_{T}\gamma$ such that the parabolic subgroup $W$ of $G(d,d,n)$, in which $w$ is a Coxeter element, is reducible, and hence 
	$W=W_{1}\times W_{2}\times\cdots\times W_{l}$ for some $l$. If for each $i\in\{1,2,\ldots,l\}$, the group $W_{i}$ is isomorphic to $G(1,1,n_{i})$ 
	for $n_{i}\leq n$, then there exists a unique increasing reduced $T$-word of $w$ with respect to the restriction of $\prec_{\gamma}$ to the 
	reflections in $T_{\gamma}\cap[\varepsilon,w]$.
\end{corollary}
\begin{proof}
	First suppose that $l=2$. In particular, we can write $w=w_{1}w_{2}$, where $w_{1}$ and $w_{2}$ commute. In view of 
	Proposition~\ref{prop:gddn_symmetric_irreducible}, each of $w_{1}$ and $w_{2}$ can be of three possible forms. Since they commute it suffices to
	consider the following cases:
	
	(i) Let $w_{1}=\bigl(\!\bigl(\colint{a}{0}\;\ldots\;\colint{b}{0}\bigr)\!\bigr)$ and 
	$w_{2}=\bigl(\!\bigl(\colint{c}{0}\;\ldots\;\colint{e}{0}\bigr)\!\bigr)$, where $a<b<e+1<d$. 
	Proposition~\ref{prop:gddn_symmetric_irreducible} implies that each of $w_{1}$ and $w_{2}$ has a unique increasing reduced $T$-word, namely
	\begin{align*}
		w_{1} & = \colref{a}{(a\!+\!1)}{0}\colref{(a\!+\!1)}{(a\!+\!2)}{0}\cdots\colref{(b\!-\!1)}{b}{0},\quad\text{and}\\
		w_{2} & = \colref{c}{(c\!+\!1)}{0}\colref{(c\!+\!1)}{(c\!+\!2)}{0}\cdots\colref{(e\!-\!1)}{e}{0},
	\end{align*}
	and the concatenation $w_{1}w_{2}$ is clearly the unique increasing reduced $T$-word of $w$. 
	
	(ii) Let $w_{1}=\bigl(\!\bigl(\colint{a}{0}\;\ldots\;\colint{b}{0}\bigr)\!\bigr)$ and 
	$w_{2}=\bigl(\!\bigl(\colint{c}{0}\;\ldots\;\colint{e}{d-1}\;\ldots\;\colint{(n\!-\!1)}{d-1}\bigr)\!\bigr)$, where $a<b<c+1<e$. Again
	Proposition~\ref{prop:gddn_symmetric_irreducible} implies that each of $w_{1}$ and $w_{2}$ has a unique increasing reduced $T$-word, namely
	\begin{align*}
		w_{1} & = \colref{a}{(a\!+\!1)}{0}\colref{(a\!+\!1)}{(a\!+\!2)}{0}\cdots\colref{(b\!-\!1)}{b}{0},\quad\text{and}\\
		w_{2} & = \colref{c}{(c\!+\!1)}{0}\colref{(c\!+\!1)}{(c\!+\!2)}{0}\cdots\colref{(e\!-\!2)}{(e\!-\!1)}{0}\\
		& \kern1cm \colref{e}{(e\!+\!1)}{0}\cdots\colref{(n\!-\!2)}{(n\!-\!1)}{0}\colref{(e\!-\!1)}{(n\!-\!1)}{d-1},
	\end{align*}
	and the concatenation $w_{1}w_{2}$ is clearly the unique increasing reduced $T$-word of $w$. 
	
	(iii) Let $w_{1}=\bigl(\!\bigl(\colint{a}{0}\;\ldots\;\colint{b}{0}\bigr)\!\bigr)$ and 
	\begin{displaymath}
		w_{2}=\bigl(\!\bigl(\colint{c}{0}\;\ldots\;\colint{e}{d-1}\;\colint{n}{s-1}\;\colint{(e\!+\!1)}{d-1}\;\ldots\;\colint{(n\!-\!1)}{d-1}\bigr)\!\bigr), 
	\end{displaymath}
	where $a<b<c+1<e$. Again Proposition~\ref{prop:gddn_symmetric_irreducible} implies that each of $w_{1}$ and $w_{2}$ has a unique increasing reduced 
	$T$-word, namely
	\begin{align*}
		w_{1} & = \colref{a}{(a\!+\!1)}{0}\colref{(a\!+\!1)}{(a\!+\!2)}{0}\cdots\colref{(b\!-\!1)}{b}{0},\quad\text{and}\\
		w_{2} & = \colref{c}{(c\!+\!1)}{0}\colref{(c\!+\!1)}{(c\!+\!2)}{0}\cdots\colref{(e\!-\!2)}{(e\!-\!1)}{0}\cdots\\
		& \kern1cm\colref{(n\!-\!2)}{(n\!-\!1)}{0}\colref{(e\!-\!1)}{(n\!-\!1)}{d-1}\colref{e}{n}{s}\colref{(n\!-\!1)}{n}{s},
	\end{align*}
	and the concatenation $w_{1}w_{2}$ is clearly the unique increasing reduced $T$-word of $w$. 

	(iv) Let $w_{1}=\bigl(\!\bigl(\colint{a}{0}\;\ldots\;\colint{b}{d-1}\;\ldots\;\colint{(c\!-\!1)}{d-1}\bigr)\!\bigr)$ and 
	\begin{displaymath}
		w_{2}=\bigl(\!\bigl(\colint{c}{0}\;\ldots\;\colint{e}{d-1}\;\ldots\;\colint{(n\!-\!1)}{d-1}\bigr)\!\bigr),
	\end{displaymath}
	where $a<b<c+1<e$. Again Proposition~\ref{prop:gddn_symmetric_irreducible} implies that each of $w_{1}$ and $w_{2}$ has a unique increasing reduced 
	$T$-word, namely 
	\begin{align*}
		w_{1} & = \colref{a}{(a\!+\!1)}{0}\colref{(a\!+\!1)}{(a\!+\!2)}{0}\cdots\colref{(b\!-\!2)}{(b\!-\!1)}{0}\\
		& \kern1cm \colref{b}{(b\!+\!1)}{0}\cdots\colref{(c\!-\!2)}{(c\!-\!1)}{0}\colref{(b\!-\!1)}{(c\!-\!1)}{d-1},\quad\text{and}\\
		w_{2} & = \colref{c}{(c\!+\!1)}{0}\colref{(c\!+\!1)}{(c\!+\!2)}{0}\cdots\colref{(e\!-\!2)}{(e\!-\!1)}{0}\\
		& \kern1cm \colref{e}{(e\!+\!1)}{0}\cdots\colref{(n\!-\!2)}{(n\!-\!1)}{0}\colref{(e\!-\!1)}{(n\!-\!1)}{d-1}.
	\end{align*}
	Since $w_{1}$ and $w_{2}$ commute, it is easy to see that there is a unique increasing reduced $T$-word of $w$, namely
	\begin{align*}
		w & = \colref{a}{(a\!+\!1)}{0}\cdots\colref{(b\!-\!2)}{(b\!-\!1)}{0}\colref{b}{(b\!+\!1)}{0}\cdots\\
		& \kern1cm\colref{(c\!-\!2)}{(c\!-\!1)}{0}\colref{c}{(c\!+\!1)}{0}\cdots\colref{(e\!-\!2)}{(e\!-\!1)}{0}\\
		& \kern1cm\colref{e}{(e\!+\!1)}{0}\cdots\colref{(n\!-\!2)}{(n\!-\!1)}{0}\colref{(b\!-\!1)}{(c\!-\!1)}{d-1}\\
		& \kern1cm\colref{(e\!-\!1)}{(n\!-\!1)}{d-1}.
	\end{align*}
	
	(v) Let $w_{1}=\bigl(\!\bigl(\colint{a}{0}\;\ldots\;\colint{b}{d-1}\;\ldots\;\colint{(c\!-\!1)}{d-1}\bigr)\!\bigr)$ and 
	\begin{displaymath}
		w_{2}=\bigl(\!\bigl(\colint{c}{0}\;\ldots\;\colint{e}{d-1}\;\colint{n}{s-1}\;\colint{(e\!+\!1)}{d-1}\;\ldots\;\colint{(n\!-\!1)}{d-1}\bigr)\!\bigr),
	\end{displaymath} 
	where $a<b<c+1<e$. Again Proposition~\ref{prop:gddn_symmetric_irreducible} implies that each of $w_{1}$ and $w_{2}$ has a unique increasing reduced 
	$T$-word, namely
	\begin{align*}
		w_{1} & = \colref{a}{(a\!+\!1)}{0}\colref{(a\!+\!1)}{(a\!+\!2)}{0}\cdots\colref{(b\!-\!2)}{(b\!-\!1)}{0}\\
		& \kern1cm \colref{b}{(b\!+\!1)}{0}\cdots\colref{(c\!-\!2)}{(c\!-\!1)}{0}\colref{(b\!-\!1)}{(c\!-\!1)}{d-1},\quad\text{and}\\
		w_{2} & = \colref{c}{(c\!+\!1)}{0}\colref{(c\!+\!1)}{(c\!+\!2)}{0}\cdots\colref{(e\!-\!2)}{(e\!-\!1)}{0}\cdots\\
		& \kern1cm\colref{(n\!-\!2)}{(n\!-\!1)}{0}\colref{(e\!-\!1)}{(n\!-\!1)}{d-1}\colref{e}{n}{s}\colref{(n\!-\!1)}{n}{s}.
	\end{align*}
	Since $w_{1}$ and $w_{2}$ commute, it is easy to see that there is a unique increasing reduced $T$-word of $w$, namely
	\begin{align*}
		w & = \colref{a}{(a\!+\!1)}{0}\cdots\colref{(b\!-\!2)}{(b\!-\!1)}{0}\colref{b}{(b\!+\!1)}{0}\cdots\\
		& \kern1cm\colref{(c\!-\!2)}{(c\!-\!1)}{0}\colref{c}{(c\!+\!1)}{0}\cdots\colref{(e\!-\!2)}{(e\!-\!1)}{0}\cdots\\
		& \kern1cm\colref{(n\!-\!2)}{(n\!-\!1)}{0}\colref{(b\!-\!1)}{(c\!-\!1)}{d-1}\colref{(e\!-\!1)}{(n\!-\!1)}{d-1}\\
		& \kern1cm\colref{e}{n}{s}\colref{(n\!-\!1)}{n}{s}.
	\end{align*}

	The case that both $w_{1}$ and $w_{2}$ are of type (iii) in Proposition~\ref{prop:gddn_symmetric_irreducible} cannot occur, since in this case
	$w_{1}$ and $w_{2}$ would not commute. The proof for $l>2$ works analogously. 
\end{proof}

\bibliography{literature}

\end{document}